\documentclass[11pt]{amsart}

\usepackage[latin1]{inputenc}
\usepackage{amsmath,amssymb}
\usepackage{color}
\usepackage{enumerate}
\usepackage{algorithmic}
\usepackage[boxed]{algorithm}
\usepackage{hyperref}
\usepackage{tikz-cd}
\usepackage{tikz}

\numberwithin{equation}{section}

\newcommand{\HilbScheme}[2]{\mathbf{Hilb}_{#1}^{#2}}

\newcommand{\Aox}{A[x_0,\mathbf x]}
\newcommand{\Ax}{A[\mathbf x]}
\newcommand{\Kox}{K[x_0,\mathbf x]}
\newcommand{\Kx}{K[\mathbf x]}

\newcommand{\Proj}{\textnormal{Proj}\,}
\newcommand{\Spec}{\textnormal{Spec}\,}

\newcommand{\sat}{{\textnormal{sat}}}
\newcommand{\reg}{\textnormal{reg}}

\newcommand{\id}{\mathfrak}
\newcommand{\hilb}{{\mathcal{H}\textnormal{ilb}}}

\newcommand{\PP}{\mathbb{P}}

\newcommand{\Mf}{\mathcal M\mathrm f}
\newcommand{\supp}{\mathrm{supp}}
\newcommand{\Ht}{\mathrm{Ht}}
\newcommand{\cN}{\mathcal{N}}
\newcommand{\Nf}{\mathrm{Nf}}
\newcommand{\SGred}{\xrightarrow{\ \id G_{\mathcal P}\ }}

\newtheorem{Lemma}{Lemma}[section]
\newtheorem{theorem}[Lemma]{Theorem}
\newtheorem{Corollary}[Lemma]{Corollary}
\newtheorem{Proposition}[Lemma]{Proposition}

\theoremstyle{definition}
\newtheorem{Notation}[Lemma]{Notation}
\newtheorem{Remark}[Lemma]{Remark}
\newtheorem{Definition}[Lemma]{Definition}
\newtheorem{Example}[Lemma]{Example}
\DeclareMathAlphabet{\mathpzc}{OT1}{pzc}{m}{it}

\begin{document}

\markboth{C.~Bertone, F.~Cioffi, M.~Roggero}
{Macaulay-Like Marked Bases}

\title{Macaulay-like marked bases}

\author[]{C. Bertone}
\email{cristina.bertone@unito.it, margherita.roggero@unito.it}
\address{Dip. di Matematica dell'Universit\`a di Torino, Torino, Italy}

\author[]{F. Cioffi}
\email{cioffifr@unina.it}
\address{Dip. di Matematica e Applicazioni dell'Universit\`a di Napoli \lq\lq Federico II\rq\rq, Napoli, Italy}

\author[]{M. Roggero}
\thanks{The three authors were partially supported by GNSAGA (INdAM, Italy). The second and third authors were partially supported by the framework of PRIN 2010-11 \lq\lq Geometria delle variet\`a\ algebriche\rq\rq, cofinanced by MIUR}



\begin{abstract}
We define marked sets and bases over a quasi-stable ideal $\id j$ in  a polynomial ring on a Noetherian $K$-algebra, with $K$ a field of any characteristic. The involved polynomials may be non-homogeneous,  but their degree is bounded from above by the maximum among the degrees of the terms in the Pommaret basis of $\id j$ and a given integer $m$. Due to the combinatorial properties of quasi-stable ideals, these bases  behave well with respect to homogenization, similarly to Macaulay bases. We prove that the family of marked bases over a given quasi-stable ideal has an affine scheme structure, is  flat and, for large enough $m$, is an open subset of a Hilbert scheme. Our main results lead to algorithms that explicitly construct such a family. 
We compare our method with similar ones and give some complexity results.
\end{abstract}

\maketitle

\section*{Introduction}

The investigation of rewriting procedures that give rise to efficient division algorithms has a longstanding tradition. A rich literature attests to the many results and applications that have been obtained in this field. Although the theory of Gr\"obner bases is maybe the most famous topic in this context, many studies have been devoted also to term-order-free procedures (see \cite{Mour,Ro,CMR13} and the references therein). Both approaches lead to  the study and the parameterization of families of ideals $I$ in a polynomial ring $S$ such that $S/I$ has a basis consisting of all the terms outside a given monomial ideal. When the ideals are homogeneous,  these families become a useful tool for the description of Hilbert scheme features (e.g.~\cite{CF,NS,RT,LR,BLR}). Under certain hypotheses, these families are interesting in themselves because they have useful properties, such as flatness \cite{BaMu,CR} or connectedness \cite{FR,LR}. 

Nevertheless, as noted in \cite[Remark 1.6]{CEVV}, for Hilbert schemes of points, every point has an open neighborhood that can be studied by suitable \lq \lq affine\rq\rq\ techniques. In the same context, a similar approach is also used in \cite[Chapter 18]{MS}. This paper is devoted to the study of a rewriting procedure in an affine framework, i.e.~the involved polynomials and ideals need not be homogeneous.  We construct  polynomial ideal bases that are well-suited for the investigation of open subsets of any Hilbert scheme (see \cite[Subsection 1.4]{BLR}, Proposition \ref{prop:immersione} and Theorem \ref{th:flatness}). To our knowledge, in the literature there are no term-order-free polynomial bases for non-homogeneous ideals of any dimension with the same properties.

Our results are inspired by previous works in a homogeneous  setting  \cite{CR,BCLR,CMR13}. Here, we consider a field of arbitrary characteristic and quasi-stable ideals. These monomial ideals play an analogous role to that of initial ideals in the Gr\"obner bases theory, but their combinatorial structure allows one to avoid the use of a term order. The starting argument to move from the homogeneous to the affine framework is the good behavior with respect to saturation of quasi-stable ideals. This behavior is inherited by the homogeneous ideals of the family corresponding to the given quasi-stable ideal (see Theorem \ref{bastaaffine}). Geometrically, this means that the corresponding schemes have no components contained in the hyperplane at infinity. As a consequence we recover the well-known result that the points of a Hilbert scheme corresponding to Cohen-Macaulay schemes form an open subset (Corollary \ref{cor:CM} and Remark \ref{rem:CM}).

We define marked sets and bases over a given quasi-stable ideal $\id j$ and a related Noetherian confluent rewriting procedure (Definitions \ref{marked affini} and \ref{relazione affine}). These definitions  depend on a positive given integer $m$ and on the Pommaret basis of $\id j$, which determine an upper bound  on  the degree of  the polynomials in every  marked set and basis.  We show that every  such a marked basis generates an ideal whose homogenization is computationally determined by the basis itself (Theorem \ref{omogbene}), with an evident similarity to the behavior of Macaulay bases \cite[Theorem 4.3.19]{KR2}. We parameterize the family of ideals that are marked over the given quasi-stable ideal by an affine scheme which we explicitly construct (Theorem \ref{thm:schemamarcatoA}). The bound on the degree of the polynomials in the marked sets guarantees that the number of parameters needed for the construction of the scheme is finite. We prove that the corresponding family is flat and that, if the bound on the degree is large enough, this scheme is an open subset of a suitable Hilbert scheme (Proposition \ref{prop:immersione} and Theorem \ref{th:flatness}). In an affine context, the property of flatness is not immediate, even for families of ideals with the same Hilbert polynomial and parameterized over an irreducible scheme (Example \ref{ex:no flat}).
  
The constructive proofs in this affine setting provide a better understanding of some results already obtained in the homogeneous one in \cite{BCLR, LR2}. For instance, Theorem \ref{critsupermin} highlights which parameters are needed in the parameterization of the given families and Theorem \ref{ro2} explains the role of the satiety of the hyperplane section of the quasi-stable ideal.  

The last section is devoted to a detailed description of the algorithms arising from our results, including a variant for the zero-dimensional case proved in Subsection \ref{conseguenze2}.
In particular, we give an example in which we investigate the irreducible components of a Hilbert scheme of curves (Example \ref{ex:tommasino}), compare our method with similar ones (Subsection \ref{subsec:comparison}) and give some complexity results. The latter consist of a bound on the degree of the polynomials involved in a reduction and a bound on the maximal length of a sequence of terms connected by a single step of reduction (Theorem \ref{th:bound}). This last result allows us to give also  an evaluation of the maximal degree of the generators of the ideal defining the affine scheme that parameterizes the marked family (Corollary \ref{cor:grado equazioni}).

%%-----------------
%% Preliminaries
%%-----------------

\section{Preliminaries}

Given a field $K$, we will denote the polynomial ring $K[x_1,\dots,$ $x_{n}]$ by $\Kx$ and the polynomial ring $K[x_0,x_1,\dots,x_n]$ by $\Kox$, letting $\mathbf x:=\lbrace x_1,\dots,x_{n}\rbrace$. For any $K$-algebra $A$, $A[\mathbf x]$ will denote the polynomial ring $A\otimes_{K} \Kx$ and $\Aox$ will denote $A\otimes_{K} \Kox$. Obviously, $\Ax$ is a subring of $\Aox$. The variables are ordered in the following way: $x_0< x_1 <\dots <x_n$. We systematically consider this setting from Section \ref{sezione:qs}  on (see for instance Definition \ref{def:qstable}).

Every $K$-algebra $A$ that we take is Noetherian with unit $1\neq 0$ and every morphism between $K$-algebras preserves the unit. 

A term is a power product $x^\alpha = x_0^{\alpha_0}\cdot\ldots\cdot x_n^{\alpha_n}$.  Let $\mathbb T_{\mathbf x}$ and $\mathbb T_{x_0,\mathbf x}$ be the set of terms in the variables $\mathbf x$ and $\{x_0\}\cup \mathbf x$, respectively. The degree of a term is $\deg(x^\alpha)=\sum \alpha_i=\vert \alpha\vert$. For every term $x^\alpha\not= 1$ we let $\max(x^\alpha):=\max\{x_i\ \vert\ \alpha_i\not=0\}$ and $\min(x^\alpha):=\min\{x_i \ \vert\ \alpha_i\not=0\}$.

For any non-zero polynomial $F \in \Aox$, the \textit{support} of $F$ is the set $\supp(F)$ of terms $x^\alpha$ in $\mathbb T_{x_0,\mathbf x}$ that appear in $F$ with a non-zero coefficient. The degree of $F$ is $\deg(F):=\max\{\deg(x^\alpha)\vert x^\alpha\in \supp(F)\}$.

For any set $\Gamma$ of polynomials, we will denote by $\Gamma_t$ the set of homogeneous polynomials of $\Gamma$ of degree $t$, by $\Gamma_{\leq t}$ the set of polynomials of $\Gamma$ of degree $\leq t$.  Furthermore, we denote by $\langle \Gamma\rangle$ the $A$-module generated by $\Gamma$. 
When $\Gamma$ is a homogeneous ideal, we denote by $\Gamma_{\geq t}$ the ideal generated by the homogeneous polynomials of $\Gamma$ of degree $\geq t$. 

We refer to \cite{Ma86} and \cite[Section 5]{KR2} for basic results about Hilbert functions. %Let $A$ be any Noetherian $K$-algebra. 
For a homogeneous ideal $I$ of $\Aox$, we denote by $H_{\Aox/I}$ the {\it Hilbert function} of $\Aox/I$ and by $P_{\Aox/I}(t)$ its {\it Hilbert polynomial}. For an ideal $\id i$ of $\Ax$, we denote by  $^aH_{\Ax/\id i}$ the {\it affine Hilbert function} of $\Ax/\id i$ and by $^aP_{\Ax/\id i}(z)$ its {\it affine Hilbert polynomial}.

A \emph{monomial} ideal is an ideal generated by terms. For a monomial ideal $J\subset \Aox$ (resp.~$\subset \Ax$) we denote   by $B_J$ the minimal set of terms generating $J$ and  by $\cN(J)$ its \emph{sous-escalier}, i.e.~the set of terms in $\Aox\setminus J$ (resp.~$\Ax\setminus J$).

The following definitions apply in every polynomial ring.

\begin{Definition}
Given a monomial ideal $J$ and an ideal $I$, a {\em $J$-reduced form modulo $I$} of a polynomial $F$ is a polynomial $F_0$ such that $F-F_0$ belongs to $I$ and $\supp(F_0)$ is contained in $\mathcal N(J)$.
If $F_0$ is the unique possible $J$-reduced form modulo $I$ of $F$, then it is called the {\em $J$-normal form modulo $I$} of $F$ and is denoted by $\Nf(F)$. (When there is no ambiguity, we omit the expression \lq\lq modulo $I$\rq\rq.)
\end{Definition}

\begin{Definition} \cite{RStu}
A {\em marked polynomial} is a polynomial $F$ together with a specified term of $\supp(F)$ that will be called {\em head term of $F$} and denoted by $\Ht(F)$. 
\end{Definition} 

Given $F\in \Aox$, we denote by $F^{a}\in \Ax$ the polynomial obtained by setting $x_0= 1$ in $F$. Conversely, if $f$ is a polynomial of $\Ax$, then we denote by $f^{h}$ the polynomial $x_0^{\deg(f)} f(x_1/x_0,\dots, x_n/x_0) \in \Aox$, which is called the \emph{homogenization} of $f$. 

Recall that the homogenization of a given ideal $\id i\subseteq \Ax$ is the ideal $\id i^h:=\{f^h\vert f \in \mathfrak i\}\subseteq \Aox$. %generated by the homogenizations $f^h$ of all the polynomials $f$ of $\id i$. 
In general, if we homogenize a set of generators of $\id i\subseteq \Ax$, we do not obtain a set of generators of $\id i^h$. However, the situation is far simpler for monomial ideals, because the homogenization of a monomial ideal $\id j$ of $R$ is $\id j^h=(B_{\id j})\cdot \Aox$. Conversely, given a homogeneous ideal $I$ in $\Aox$ generated by a finite set of homogeneous polynomials $\{F_1,\ldots,F_s\}$, we define  the ideal $I^a$ in $\Ax$ as the ideal generated by the set $\{F^a_1,\ldots,F^a_s\}$. 

The \emph{saturation} $I^\sat$ of a  homogeneous  ideal $I$ in a polynomial ring is the saturation of $I$ with respect to the irrelevant maximal ideal.  
The ideal $I$ is \emph{saturated} if $I=I^{\sat}$, and is $m$-saturated if $I_t=({I^\sat})_t$ for every $t\geq m$. The \emph{satiety} $\sat(I)$ of $I$ is the smallest $m$ for which $I$ is $m$-saturated.
A homogeneous polynomial $F$ in $\Aox$ (resp. $\Ax$) is \emph{generic for $I$} if $F$ is a non-zero divisor in $\Aox/I^\sat$ (resp.~$\Ax/I^{\sat}$). If $I$ is an Artinian ideal, then (we say that) every homogeneous polynomial $F$ is generic for $I$ (see \cite[Definition (1.5)]{BS})

The \emph{regularity} $\reg(I)$ of a homogeneous ideal $I\subseteq \Aox$ (resp.~$\subseteq \Ax$) is the maximal degree of the generators of its syzygies modules. The ideal $I$ is $m$-regular if $m\geq \reg(I)$. 

%%-------------------------------------------------
%% Quasi-stable, stable and strongly stable ideals
%% ------------------------------------------------

\section{Quasi-stable ideals and Pommaret bases}
\label{sezione:qs}

In the present section, we consider monomial ideals either in $\Ax$ or in $\Aox$, unless a different assumption is stated explicitly.

\begin{Definition}\label{def:qstable}
A monomial ideal $J$ is \emph{quasi-stable} if there is $s\geq 0$ such that $\displaystyle\frac{x_j^s x^\alpha}{\min(x^\alpha)}$ belongs to $J$, for every $x^\alpha \in J$ and for every $x_j>\min(x^\alpha)$.
\end{Definition}

Now, we recall some of several properties characterizing the quasi-stable ideals. 

\begin{theorem}\label{thm:BorelEquiv}
Let $\ell$ be in $\{0,\dots,n-1\}$, $J$ be a monomial ideal in a polynomial ring $R$ over $K$ with variables  $\{x_\ell,\dots, x_n\}$,   and $d$ be the Krull dimension of $R/J$. Recalling that we ordered the variables so that $x_i<x_{i+1}$ for every $i\in\{\ell,\dots,n-1\}$, the following conditions are equivalent:
\begin{enumerate}[(i)]
\item \label{BorelEquiv_i} $J$ is quasi-stable;
\item \label{BorelEquiv_ii} for each term $x^\alpha\in J$, for all non-negative integer $m$ and for all integers $i,j\in\{\ell,\dots,n\}$ with $i<j$ so that $x^\alpha$ is divisible by $x_i^m$, there exists $s\geq 0$ such that $\displaystyle\frac{x_j^s x^\alpha}{x_i^m}\in J$;
\item \label{BorelEquiv_iii} for each term $x^\alpha\in J$ and for all integers $i,j\in\{\ell,\dots,n\}$ with $i<j$ so that $x^\alpha$ is divisible by $x_i$, there exists $s\geq 0$ such that $\displaystyle\frac{x_j^s x^\alpha}{x_i^{\alpha_i}}\in J$;
\item \label{BorelEquiv_v} $(J:x_i^\infty)=(J:(x_i,\dots,x_n)^\infty)$, for every $x_\ell \leq x_i\leq x_n$ or $\ell\leq i \leq n$.
\item  \label{BorelEquiv_vi} the smallest variable $x_\ell$ is not a zero divisor for $R/J^{\sat}$ and, for all $\ell\leq j\leq \ell+d-1$, the variable $x_{j+1}$ is not a zero divisor for $R/(J,x_\ell,\dots,x_j)^{\sat}$.
\end{enumerate}
\end{theorem}

\begin{proof}
An explicit proof of the equivalence between items \eqref{BorelEquiv_i} and  \eqref{BorelEquiv_ii} is in \cite[Theorem 3.4]{B14}.
The equivalences among items \eqref{BorelEquiv_ii}, \eqref{BorelEquiv_iii} and \eqref{BorelEquiv_v} are proved in \cite[Proposition 2.2]{HPV03}. The equivalence between items \eqref{BorelEquiv_v} and \eqref{BorelEquiv_vi} is proved in \cite[Proposition 3.2]{BG06}.
\end{proof}

\begin{Corollary}\label{lem:0qs}
Let $J$ be a monomial ideal of a polynomial ring $R$ over the field $K$ with variables  $\{x_\ell,\dots, x_n\}$, $x_j<x_{j+1}$ and $d$ the Krull dimension of $R/J$. If $B_J$ consists of terms non-divisible by any of the smallest $d$ variables, then it is quasi-stable. In particular, if $R/J$ is  Artinian then it is quasi-stable.
\end{Corollary}

\begin{proof}
Let $\mathfrak m$ be the maximal irrelevant ideal of the ring $R$. If $d=0$, i.e.~$J$ is Artinian, then there exists a positive integer $s$ such that $J_m=\mathfrak m_m$, for every $m\geq s$. Then $J$ is quasi-stable, by Theorem \ref{thm:BorelEquiv} \eqref{BorelEquiv_ii}. If $d>0$, then the smallest $d$ variables form a regular sequence of linear forms for $R/J$, hence $R/J$ is a Cohen-Macaulay ring and we can apply Theorem \ref{thm:BorelEquiv} \eqref{BorelEquiv_vi}.
\end{proof}

\begin{Corollary}\label{negativo} Let $\id j\subseteq \Ax$ and $J\subseteq \Aox$ be quasi-stable ideals.
\begin{enumerate}[(i)]
\item If $x^\gamma$ is a term in $\mathbb T_{\mathbf x}$, then $x^\gamma$ belongs to $\cN(\id j)\subseteq \mathbb T_{\mathbf x}$ if and only if $x_0^{r}x^\gamma$ belongs to $\cN(\id j^{h})\subseteq \mathbb T_{x_0,\mathbf x}$ for all $r\geq 0$. Moreover, if $f$ is a polynomial in $\Ax$ with $\deg(f)=t$, then
\[
f\in \langle\cN(\mathfrak j)_{\leq t}\rangle  \Leftrightarrow \forall \ r\geq 0,\  x_0^r f^h\in \langle\cN(\id j^h)_{t+r}\rangle.
\]
\item If $J\subseteq \Aox$ is saturated, then a term $x^\gamma$ belongs to $\cN(J) \subseteq \mathbb T_{x_0,\mathbf x}$ if and only if $(x^{\gamma})^a$ belongs to $\cN(J^a)\subseteq \mathbb T_{\mathbf x}$. Moreover, if $F$ is a homogeneous polynomial in $\Aox$, then
\[
F\in\langle \cN(J)_{t}\rangle\Leftrightarrow F^a\in \langle \cN(J^a)_{\leq t}\rangle.
\]
\end{enumerate}
\end{Corollary}

\begin{proof}
It is enough to recall that $B_{\id j}=B_{\id j^h}$ and $B_J=B_{J^a}$, because $J$ is saturated, hence $J=(J:x_0^{\infty})$ by Theorem \ref{thm:BorelEquiv}.
\end{proof}

Quasi-stable ideals can be characterized also by Pommaret bases.

\begin{Definition}\cite{Seiler2009I}
The \emph{Pommaret cone} of a term $x^\alpha$ is the set of terms $\mathcal C_{\mathcal P}(x^\alpha)=\lbrace x^\delta x^\alpha \vert \max(x^\delta)\leq \min(x^\alpha)\rbrace$. Given a finite set $M$ of terms, its \emph{Pommaret span} is $\cup_{x^\alpha\in M} \mathcal C_{\mathcal P}(x^\alpha)$.

The finite set of terms $M$ is a \emph{weak Pommaret basis} of $(M)$ if the Pommaret span of $M$ is the set of terms in the ideal $(M)$  and is a $\emph{Pommaret basis}$ if the Pommaret cones of the terms in $M$ are pairwise disjoint. 
\end{Definition}

\begin{theorem}\label{qsPB}\cite[Definition 4.3 and Proposition 4.4]{Seiler2009II} \ 
A monomial ideal $J$ is quasi-stable if and only if it has a Pommaret basis.
\end{theorem}

It is noteworthy that the Pommaret basis of a quasi-stable ideal is unique (e.g.~\cite[Proposition 2.11 and Algorithm 2]{Seiler2009I}). Hence, for every $x^\gamma$ belonging to a quasi-stable ideal $J$, there is a unique $x^\alpha \in \mathcal P(J)$ such that $x^\gamma \in \mathcal C_{\mathcal P}(x^\alpha)$. In this case, observe that $\max(x^\gamma/x^\alpha)\leq \min(x^\alpha)$. 

The following statement collects some further well-known results.

\begin{Lemma}\label{lemma:importante}
Let $J$ be a quasi-stable ideal  and $\mathcal P(J)$ be its Pommaret basis.
\begin{enumerate}[(i)]
\item\label{importante_00} the regularity of $J$ is the maximum among the degrees of terms in $\mathcal P(J)$.
\item \label{importante_0} the satiety of $J$ is the maximum among the degrees of terms in $\mathcal P(J)$ which are divisible by the smallest variable in the polynomial ring.
\item\label{importante_i} $x^\alpha \in J\setminus \mathcal P(J) \Rightarrow x^\alpha / \min(x^\alpha) \in J$.
\item\label{importante_ii} $x^\beta \in \mathcal N(J)$ and $x_ix^\beta \in J$ $\Rightarrow$ either $x_ix^\beta \in \mathcal P(J)$ or $x_i>\min(x^\beta)$.
\item \label{importante_iii} $x^\eta \in \mathcal N(J)$ and $x^\eta\cdot x^\gamma \in J$ with $x^\eta x^\gamma\in \mathcal C_{\mathcal P}(x^\alpha)$ for $x^\alpha \in \mathcal P(J)$ $\Rightarrow$  $(x^\eta x^\gamma)/x^\alpha <_{lex} x^\gamma$.
\item \label{importante_iv} If $J$ is saturated, then the smallest variable in the polynomial ring does not divide any term in $\mathcal P(J)$.
\end{enumerate}
\end{Lemma}

\begin{proof}  Items \eqref{importante_00} and  \eqref{importante_0} are proved in \cite[Theorem 9.2, Corollary 10.2]{Seiler2009II}. 
Items \eqref{importante_i} and \eqref{importante_ii} follow by the definition of Pommaret basis. For \eqref{importante_iii}, we can use exactly the same arguments that have been applied in the proof of \cite[Lemma 2.5]{BCLR}. Item \eqref{importante_iv} is proved in \cite[Lemma 4.11]{Seiler2009II}. 
\end{proof}

\begin{Remark}\label{rem:BPaff-hom}
Let $\id j\subseteq \Ax$ be a quasi-stable ideal with Pommaret basis $\mathcal P(\id j)$. Then, the homogenization $\id j^h=(B_{\id j})\cdot \Aox$ is a saturated  quasi-stable ideal and its Pommaret basis is $\mathcal P(\id j^h)=\mathcal P(\id j)$.
For a saturated quasi-stable ideal $J\subseteq \Aox$ we have $B_J\subset \Ax$, because $J=J^{\sat}=(J:x_0^{\infty})$, by Theorem \ref{thm:BorelEquiv} \eqref{BorelEquiv_v}. Moreover, the ideal $J^a=(B_J)\cdot \Ax$ is quasi-stable and its Pommaret basis is exactly $\mathcal P(J)$, by Lemma \ref{lemma:importante}\eqref{importante_iv}.
\end{Remark}

If $J$ is quasi-stable and $s\geq 0$, then $J_{\geq s}$ is quasi-stable too \cite[Lemma 4.6]{Seiler2009II}. Here we describe the Pommaret basis of $J_{\geq s}$ starting from that of $J$.

\begin{Lemma}\label{lem:star}\cite[Proposition 4.7]{CMR13} The Pommaret basis of a quasi-stable ideal $J$ is 
\begin{equation*}
\mathcal P(J)=\Bigl\{x^\alpha \in J : \dfrac{x^\alpha}{\min(x^\alpha)} \in\mathcal N(J)\Bigr\}.
\end{equation*}
\end{Lemma}

\begin{Proposition}\label{prop:PBtronc}
Let $J$ be a quasi-stable ideal.
\begin{enumerate}[(i)]
\item \label{prop:PBtronc_i}
$\mathcal P(J_{\geq s})=\mathcal P(J)_{\geq s+1}\cup \left(\bigcup_{x^\alpha\in \mathcal P(J)_{\leq s}}\mathcal C_{\mathcal P}(x^\alpha)_s\right)$, for every $s\geq 0$.
\item \label{prop:PBtronc_ii} If $J$ is $m$-saturated, then $\mathcal P(J)_\ell=\mathcal P(J^\sat)_\ell$, for every $\ell\geq m+1$.
\end{enumerate}
\end{Proposition}

\begin{proof}%\ 
%\begin{enumerate}[(i)]
%\item 
(i) We consider the description of $\mathcal P(J_{\geq s})$ given in  Lemma \ref{lem:star}. First, we observe that the set of terms $x^\beta \in \mathcal P(J_{\geq s})$ with $\vert \beta \vert>s$ is exactly the set $\mathcal P(J)_{\geq s+1}$. If $x^\beta$ belongs to $\mathcal P(J_{\geq s})$ with $\vert \beta \vert=s$, then there are two possibilities: either $x^\beta$ belongs to $\mathcal P(J)_s$ or $x^\beta\in \mathcal C_{\mathcal P}(x^\alpha)$, with $x^\alpha \in \mathcal P(J)_{<s}$ and $\deg(x^\beta/x^\alpha)\geq 1$. In the latter case, $x^\beta$  belongs to $\mathcal C_{\mathcal P}(x^\alpha)_s$. Conversely, for every $x^\alpha \in \mathcal P(J)_{\leq s}$  and for every $x^\beta\in \mathcal C_{\mathcal P}(x^\alpha)_s$, $x^\beta/\min(x^\beta)$  does not belong to $J_{\geq s}$, hence $x^\beta$ belongs to $\mathcal P(J_{\geq s})$ by Lemma \ref{lemma:importante}\eqref{importante_ii}.

(ii) Being $J$ $m$-saturated, we have $J_{\geq t}=({J^\sat})_{\geq t}$ for every $t\geq m$, hence $\mathcal N(J)_t=\mathcal N(J^{\sat})_t$ for every $t\geq m$. Then, it is enough to use  Lemma \ref{lem:star}.
%\end{enumerate}
\end{proof}

%%--------------------------------------------
%% Saturation in marked families 
%%--------------------------------------------

\section{Saturation in marked families}

In the present section we restrict our attention on homogeneous polynomials and homogeneous ideals in $\Aox$.

\begin{Definition}\label{def:cose marcate} \cite[Definition 5.1]{CMR13}
Let $J\subset \Aox$ be a quasi-stable ideal.

A {\em $\mathcal P(J)$-marked set} (or marked set over $\mathcal P(J)$) $G$ is a set of homogeneous monic marked polynomials $F_\alpha$ in $\Aox$ such that  the head terms $Ht(F_\alpha)=x^\alpha$ are pairwise different and form the Pommaret basis $\mathcal P(J)$ of $J$, and $\supp(F_\alpha-x^\alpha)\subset \mathcal N(J)$.

A $\mathcal P(J)$\emph{-marked basis} (or marked basis over $\mathcal P(J)$) $G$ is a $\mathcal P(J)$-marked set such that $\mathcal N(J)$ is a basis of $\Aox/(G)$ as an $A$-module, i.e.~$\Aox=(G)\oplus \langle \cN (J) \rangle$ as an $A$-module.

The {\em $\mathcal P(J)$-marked family} (or marked family over $\mathcal P(J)$) $\Mf(\mathcal P(J))$ is the set of all homogeneous ideals $I$ that are generated by a $\mathcal P(J)$-marked basis.
\end{Definition} 

A consequence of Definition \ref{def:cose marcate} is that a $\mathcal P(J)$-marked set $G$ is a $\mathcal P(J)$-marked basis if and only if every polynomial of $\Aox$ has a $J$-normal form modulo $(G)$. 

\begin{Definition}\cite[Definitions 5.3 and 5.6]{CMR13}\label{relazione omogenea}
Let $J\subseteq \Aox$ be a quasi-stable ideal and $G=\{f_\alpha\}_{x^\alpha\in\mathcal P(J)}$ a $\mathcal P(J)$-marked set. For each degree $\ell$ such that $J_\ell\not=\{0\}$, we denote by $G^{(\ell)}$ the set of homogeneous polynomials
$$G^{(\ell)}=\{x^\eta F_\alpha  \ \vert \ x^\eta x^\alpha \in J_\ell \text{ and }x^\eta x^\alpha  \in \mathcal C_{\mathcal P}(x^\alpha) \}\subseteq (G)_\ell$$
that are marked on the terms of $J_\ell$ in the natural way $Ht(x^\eta F_\alpha )=x^\eta x^\alpha$. 

For every $\ell\geq 0$, we will denote by $\xrightarrow{\ G^{(\ell)} \ }$ the reflexive and transitive closure of the following reduction relation on $\Aox_\ell$: $F$ is in relation with $F'$ if $F'=F-\lambda x^\eta F_\alpha$, where $x^\eta F_\alpha$ belongs to $G^{(\ell)}$ and $\lambda\neq 0_A$ is the coefficient of the term $x^{\eta+\alpha}$ in $F$.
We will write $F \xrightarrow{\ G^{(\ell)}\ }_\ast F_0$ if   $F\in \Aox_\ell$, $F \xrightarrow{\ G^{(\ell)}\ } F_0$ and $\supp(F_0) \subseteq \mathcal N(J)_{ \ell}$.
\end{Definition}

\begin{theorem}\label{th:varie CMR}\cite[Lemma 5.8, Theorems 5.9 and 5.13 and Corollary 5.11]{CMR13} 
Let $J\subseteq \Aox$ be a quasi-stable ideal and $G$ a $\mathcal P(J)$-marked set. Then, the reduction relation $\xrightarrow{\ G^{(\ell)}\ }$ is Noetherian and confluent and the followings are equivalent:
\begin{enumerate}[(i)]
\item $G$ is a $\mathcal P(J)$-marked basis;
\item \label{varie CMR_ii}$\langle G^{(\ell)}\rangle =(G)_\ell$, for every $\ell$ such that $J_\ell\not=\{0\}$;
\item \label{varie CMR_iii}$F \xrightarrow{\ G^{(\ell)}\ }_\ast 0$, for every $\ell\geq 0$ and  for every $F\in (G)_{\ell}$;
\item \label{varie CMR_iv} $x_iF_\alpha \xrightarrow{\ G^{(\ell+1)}\ }_\ast 0$ for every $\ell\geq 0$, for every $F_\alpha \in G_{ \ell}$ and for every $x_i > \min(x^\alpha)$.
\end{enumerate}
\end{theorem}

We now show that, if $J\subseteq \Aox$ is a quasi-stable ideal, then the satiety of $J$ is a bound from above for  the satiety of every ideal $I\in\Mf({\mathcal P}(J))$ and the scheme $\Proj(\Aox/I)$ has no component at infinity. The following result is analogous to \cite[Lemma 4.2]{BCLR}. 

\begin{Lemma}\label{lemma1}
Let $J\subseteq \Aox$ be an $m$-saturated quasi-stable ideal,
$F\in \Aox$ a polynomial of degree $\ell\geq m$ and $G$ a $\mathcal P(J)$-marked set. Then
$$F\in \langle G^{(\ell)}\rangle \Leftrightarrow x_0 F\in \langle G^{(\ell+1)}\rangle.$$
\end{Lemma}

\begin{proof}
If $F\in \langle G^{(\ell)}\rangle$ then $x_0\cdot F\in \langle G^{(\ell+1)}\rangle$ by definition.
Conversely, if $x_0 F\in \langle G^{(\ell+1)}\rangle\subset (G)$ then $x_0\cdot F \xrightarrow{\ G^{(\ell+1)}\ }_\ast 0$ by Theorem \ref{th:varie CMR}. Moreover, observe that for every term $x_0x^\gamma\in J\cap\supp(x_0 F)$, we have $x_0x^\gamma \notin \mathcal P(J)$. Indeed, since $\deg(x_0x^\gamma)=\ell+1 \geq m+1$, if $x_0x^\gamma\in \mathcal P(J)$, then $x_0x^\gamma\in \mathcal P(J^\sat)$, by Proposition \ref{prop:PBtronc}\eqref{prop:PBtronc_ii}. But every term  in $\mathcal P(J^\sat)$ is not divisible by $x_0$, by Lemma \ref{lemma:importante} \eqref{importante_iv}. 
Hence, in order to reduce every term $x_0x^\gamma$ by $\xrightarrow{\ G^{(\ell+1)}\ }$ we use always polynomials divisible by $x_0$, i.e.~polynomials of type $x_0x^\eta F_\alpha $, with $x^\eta F_\alpha$ belonging to $G^{(\ell)}$. 
\end{proof}

\begin{theorem}\label{bastaaffine}
Let $J\subseteq \Aox$ be a quasi-stable ideal and $I$ an ideal in $\Mf({\mathcal P}(J))$. Then:
\begin{enumerate}[(i)]
\item $x_0$ is generic for $I$;
\item if $J$ is saturated, then $I$ is saturated;
\item \label{bastaffineiii} $\sat(I)\leq \sat(J)$.
\end{enumerate}
Moreover, $(I:x_0^\infty)=I^\sat$.
\end{theorem}

\begin{proof}
Let $G$ be the $\mathcal P(J)$-marked basis of $I$. Then, by Theorem \ref{th:varie CMR}\eqref{varie CMR_ii} and, assuming that $J$ is $m$-saturated, by Lemma \ref{lemma1}, we obtain  
$$F\in (I:x_0)_t \Rightarrow Fx_0 \in I_{t+1}=\langle G^{(t+1)}\rangle \Rightarrow F\in \langle G^{(t)}\rangle=I_t$$
for every $t\geq m$, i.e.~$(I:x_0)_t=I_t$ for every $t\geq m$.
By \cite[Lemma (1.6)]{BS}, this is equivalent to the fact that $I$ is $m$-saturated too and $x_0$ is generic for $I$.
\end{proof}

\begin{Remark}
If $K$ is not an infinite field, there may not be a generic linear form for an ideal $I$. However, even if $K$ is finite, $x_0$ is a generic linear form for a  quasi-stable ideal $J$, and by Theorem \ref{bastaaffine}, it is generic for $I\in \Mf({\mathcal P}(J))$ too.
\end{Remark}

\begin{Corollary}\label{sattronc}
Let $J\subseteq \Aox$ be a  saturated quasi-stable ideal and $m$ be a non-negative integer. If $I$ is a non-saturated ideal of $\Mf({\mathcal P}(J_{\geq m}))$, then $\sat(I)=m$ and $I=(I^{\sat})_{\geq m}$.
\end{Corollary}

\begin{proof}
Since $I$ is not saturated, then by Theorem \ref{bastaaffine} the ideal $J_{\geq m}$ is not saturated and $\sat(J)=m$.  Since $m$ is also the initial degree of the ideal $I$, we also have that $\sat(I)\geq m$. By Theorem \ref{bastaaffine}(\ref{bastaffineiii}), we have $\sat(I)=m$ and, in particular, $I=(I^{\sat})_{\geq m}$.
\end{proof}

If $J$ is a quasi-stable ideal, in general the family $\Mf({\mathcal P}(J))$ contains saturated ideals, even if $J$ is not saturated, as the following example shows.

\begin{Example}\label{ex:saturati}
We consider $S=K[x_0,x_1,x_2]$ over a field $K$ of characteristic zero and the saturated quasi-stable ideal $J=(x_2^2,x_1^3x_2,x_1^4)$. In this case $\mathcal P(J)=B_J$  and $S/J$ has Hilbert polynomial $7$. Geometrically, $\Proj(S/J)$ is a non-reduced scheme of length 7. The family $\Mf({\mathcal P}(J_{\geq 3}))$ is a dense open subset of $\hilb_7^2$ (see \cite{BLR}), which has only one component. Since 7 general points in $\PP^2$ are not on a conic, there is an open subset of $\Mf({\mathcal P}(J_{\geq 3}))$ made up of saturated ideals. 
More explicitly, we consider the $\mathcal P(J_{\geq 3})$-marked sets:
\[G_{u,v}=\mathcal P(J_{\geq 3})\setminus\{x_2^2x_0,x_1^4\}\cup\{x_{{2}}^{2}x_{{0}}-ux_{{1}}^{2}x_{{2}},x_{{1}}^{4}+vux_{{1}}^{2}x_{{0}}^{2}-vx_{{2}}x_{{0}}^{3}\}, \quad u,v \in K\]
and define $I_{u,v}:=(G_{u,v})K[x_0,x_1,x_2]$. Since $\dim (I_{u,v})_t=\dim J_t$ for every $t\geq 3$, $I_{u,v}$ belongs to $\Mf({\mathcal P}(J_{\geq 3}))$ and is saturated, for every $u,v \in K$.
\end{Example}

\begin{Corollary}\label{cor:CM}
Let $J$ be a monomial ideal of $K[x_0,\mathbf x]$ and $d$ the Krull dimension of $K[x_0,\mathbf x]/J$. If $B_J$ consists of terms in $K[x_d,\dots,x_n]$, then $J$ is quasi-stable and $K[x_0,\mathbf x]/I$ is Cohen-Macaulay for every ideal $I$ of $\Mf({\mathcal P}(J))$.
\end{Corollary}

\begin{proof}
It is enough to apply Corollary \ref{lem:0qs} and Theorem \ref{bastaaffine}.
\end{proof}

\begin{Remark}\label{rem:CM}
By Corollary \ref{cor:CM} we recover the well-known fact that the points of a Hilbert scheme corresponding to arithmetically Cohen-Macaulay schemes form an open subset, at least in characteristic $0$. Indeed, up to a suitable coordinate change, every ideal $I$ with $R/I$ Cohen-Macaulay belongs to a family $\Mf({\mathcal P}(J))$, for an ideal $J$ satisfying the same hypotheses of Corollary \ref{cor:CM} with $d\geq 1$. For example, it is sufficient to consider the generic initial ideal of $I$ with respect to the graded reverse lexicographic term order. If $d=1$, then all the points of the Hilbert scheme correspond to arithmetically Cohen-Macaulay schemes. If $d>1$, then $\Mf({\mathcal P}(J))$ can be embedded in a Hilbert scheme as an open subscheme by \cite[Subsection 1.4 and Theorem 3.1]{BLR} and we conclude. We can also observe that every quasi-stable saturated ideal $J\subset \Aox$ such that $\Aox/J$ is Cohen-Macaulay is the so-called double-generic initial ideal of the irreducible components containing $J$ in a Hilbert scheme (see \cite[Proposition 4(b) and Definition 5]{BCR}).
\end{Remark}
%{\color{blue}Cose sottintese nel precedente Remark.
%
%Sia $X$ uno schema aritmeticamente Cohen-Macauly (aCM) di dimensione $k$. Se $I$ \`{e} l'ideale saturato che definisce $X$, allora l'anello $S/I$ ha dimensione di Krull $d=k+1 \geq 1$. A meno di un cambiamento di variabili, possiamo supporre che l'ideale iniziale $J$ di $I$ rispetto al degrevlex sia generato da termini nelle variabili $x_{k+1},\dots,x_n$. Allora, per il Corollary \ref{cor:CM}, l'ideale monomiale $J$ \`{e} quasi-stable (e saturato) e $I$ appartiene alla famiglia marcata $\Mf(J)$.
%
%Se $d>1$, allora abbiamo $\rho=0$ per cui $\Mf(J)\simeq \Mf(J_{\geq m})$ per ogni $m\geq 0$, e possiamo dire che $\Mf(J)$ si immerge nello schema di Hilbert come un aperto.
%
%Se $d=1$, tutti i punti dello schema di Hilbert sono aCM.
%}

%%---------------------------------------
%% The affine setting and homogenization
%%---------------------------------------

\section{The affine setting and homogenization}
\label{sec:affine}

Let $\mathfrak j\subseteq \Ax$ be a quasi-stable ideal and $m$ a non-negative integer. Recall that $J:=B_{\mathfrak j}\cdot \Aox$ is a saturated quasi-stable ideal in $\Aox$ and $J_{\geq m}={(\id j)^h}_{\geq m}$.

\begin{Definition}\label{marked affini}
A {\em $[\mathcal P(\mathfrak j),m]$-marked set} $\mathfrak G$ is a set of monic marked  polynomials $f_\alpha$ of $\Ax$ such that the head terms $\Ht(f_\alpha)=x^\alpha$ are pairwise different and form the Pommaret basis $\mathcal P(\id j)$ of $\id j$, and $\supp(f_\alpha-x^\alpha)\subseteq \mathcal N(\id j)_{\leq t}$ with $t=\max\{m,\vert \alpha \vert\}$. 

A $[\mathcal P(\mathfrak j),m]$-marked set $\mathfrak G=\{f_\alpha\}_{x^\alpha \in \mathcal P(\id j)}$ is a {\em $[\mathcal P(\mathfrak j),m]$-marked basis} if there is a $\mathcal P(J_{\geq m})$-marked basis $G$ with suitable integers $k_\alpha$ such that $x_0^{k_\alpha}f_\alpha^h$ belongs to $G$, for every $x^\alpha \in \mathcal P(\id j)$.

The {\em $[\mathcal P(\id j),m]$-marked family} $\Mf({\mathcal P}(\id j),m)$ is the set of all the ideals $I\subseteq \Ax$ that are generated by a $[\mathcal P(\id j),m]$-marked basis.
\end{Definition}

%When $\mathfrak G=\{f_1,\dots,f_t\}$ is a $[\mathfrak j,m]$-marked set, 
We  consider the following relation, whose definition involves a $[\id j,m]$-marked set.

\begin{Definition}\label{relazione affine}
Let $\id G=\{f_\alpha\}_{x^\alpha\in\mathcal P(\id j)}$ be a $[\mathcal P(\id j),m]$-marked set. We denote by $\id G_{\mathcal P}$ the set of polynomials
$$\id G_{\mathcal P}=\{x^\delta f_\alpha  \ \vert f_\alpha \in \id G \text{ and } \ x^\delta x^\alpha \in \mathcal C_{\mathcal P}(x^\alpha) \}\subseteq (\id G)$$
that are marked on the terms of $\id j$ in the natural way $Ht(x^\delta f_\alpha )=x^\delta x^\alpha$. 

We will denote by $\xrightarrow{\ \id G_{\mathcal P}\ }$ the reflexive and transitive closure of the  following reduction relation on $\Ax$: $f$ is in relation with $f'$ if $f'=f-\lambda x^\delta f_\alpha $, where $x^\delta f_\alpha \in \id G_{\mathcal P}$ and $\lambda\neq 0_A$ is the coefficient of the term $x^{\alpha+\delta}$ in $f$.
We will write $f \xrightarrow{\ \id G_{\mathcal P}\ }_\ast h$ if $f \xrightarrow{\ \id G_{\mathcal P}\ } h$ and $h \in \langle \mathcal N(\id j)\rangle$.
\end{Definition}

\begin{Proposition} \label{th:reductionAffineSet}
Let $\id G\subset \Ax$ be a $[\mathcal P(\id j),m]$-marked set. 
\begin{enumerate}[(i)]
\item \label{rAS_i} The reduction relation $\xrightarrow{\ \id G_{\mathcal P}\ }$ is Noetherian.
\item\label{rAS_ii} For every polynomial $f\in \Ax$ there is $\bar f\in \langle \mathcal N(\id j)\rangle$ such that $f \xrightarrow{\ \id G_{\mathcal P}\ }_\ast \bar f$. The polynomial $\bar f$ is a $\id j$-reduced form of $f$ modulo the ideal $(\id G)$.
\item \label{rAS_iii} $\Ax=\langle \id G_{\mathcal P}\rangle \oplus \langle \mathcal N(\id j)\rangle$.
\item \label{rAS_iv} The reduction relation $\xrightarrow{\ \id G_{\mathcal P}\ }$ is confluent.
\end{enumerate}
\end{Proposition}
\begin{proof} %\ 
%\begin{enumerate}[(i)]
%\item 
(i) It is enough to observe that if $\xrightarrow{\  \id G_{\mathcal P}\ }$ was not Noetherian, by Lemma \ref{lemma:importante}\eqref{importante_iii}  applied to $\id j$ we would find a strictly descending  infinite chain of terms w.r.t.~the lex term order, that is impossible. 

%\item 
(ii) By the Noetherianity of the reduction relation $\xrightarrow{\ \id G_{\mathcal P}\ }$, there is always a finite number of steps of the reduction $\xrightarrow{\ \id G_{\mathcal P}\ }$ from $f$ to a polynomial $\bar f$  that is reduced w.r.t.~$\xrightarrow{\ \id G_{\mathcal P}\ }$. By Definition \ref{relazione affine}, we must have $\bar f\in \langle \mathcal N(\id j)\rangle$ and $f-\bar f\in \langle \id G_{\mathcal P}\rangle \subseteq (\id G)$. Hence, $\bar f$ is a $\id j$-reduced form of $f$ modulo the ideal $(\id G)$.

%\item 
(iii) By \eqref{rAS_ii} we obtain $\Ax=\langle \id G_{\mathcal P}\rangle + \langle \mathcal N(\id j)\rangle$. So, it is sufficient to show that $\langle \id G_{\mathcal P}\rangle \cap \langle\mathcal N(\id j)\rangle=\{0\}$.
Let $f$ be a non-zero polynomial belonging to $\langle \id G_{\mathcal P}\rangle \cap \langle \mathcal N(\id j)\rangle$. 
Since $f$ belongs to $\langle \id G_{\mathcal P}\rangle$, we can write $f=\sum \lambda_{\delta\alpha} x^\delta f_\alpha$, for suitable $\lambda_{\delta\alpha}\in A\setminus \{0\}$ and $x^\delta f_\alpha\in \id G_{\mathcal P}$ (we can obtain such a writing by applying $\xrightarrow{\ \id G_{\mathcal P}\ }$), and, in particular, every $x^\delta f_\alpha$ appears only once. Since $f$ belongs to $ \langle \mathcal N(\id j)\rangle$, for every $x^\delta f_\alpha$ appearing in the writing $f=\sum \lambda_{\delta\alpha} x^\delta f_\alpha$, the terms in $\supp(x^\delta f_\alpha)\cap \id j$ cancel with some terms of other polynomials $x^{\delta'} f_{\alpha'}\in \id G_{\mathcal P}$.  Among the terms $x^\delta \Ht(f_\alpha)$ appearing in the writing of $f$, consider the one having maximal $x^\delta$ according to the lex term order. This term $x^\delta x^\alpha$ should cancel with some term in $x^{\delta'}(f_{\alpha'}-x^{\alpha'})$, for some $x^{\delta'}f_{\alpha'}$ appearing in the summation. Suppose that for some $x^\eta \in \supp(f_{\alpha'}-x^{\alpha'})\subseteq \mathcal N(\id j)$ we have that $x^{\delta'}x^\eta= x^\delta x^\alpha\in \mathcal C_{\mathcal P}(x^\alpha)$. Then, by Lemma \ref{lemma:importante} \eqref{importante_iii}, it follows $x^{\delta'}>_{lex}x^\delta$,  contradicting the maximality of $x^\delta$. Hence, the term $ x^\delta x^\alpha$  cannot be cancelled by any other term. We can conclude that $f=0$.

%\item 
(iv) This is a consequence of item \eqref{rAS_iii}.
%\end{enumerate}
\end{proof}

Although the definition of $[\mathcal P(\id j),m]$-marked basis $\id G$ in $A[\mathbf x]$ relies on the existence of a homogeneous marked basis in $\Aox$, the following two theorems give an equivalent condition that only involves the $\id j$-reduced forms modulo $(\id G)$ and their degrees.

\begin{theorem}\label{th:markedBasisA1}
If $\id G$ is a $[\mathcal P(\id j),m]$-marked basis then every polynomial $f\in \Ax$ has $\id j$-normal form $\Nf(f)$ modulo $(\id G)$ with $\deg(\Nf(f))\leq \max\{\deg(f),m\}$. 
\end{theorem}

\begin{proof}
By Definition \ref{marked affini}, there are a $\mathcal P(J_{\geq m})$-marked basis $G$ and suitable integers $k_\alpha$ such that $x_0^{k_\alpha}f_\alpha^h$ belongs to $G$, for every $x^\alpha \in \mathcal P(\id j)$.  Thus, for every element $x^\delta f_\alpha$ in $\id G_{\mathcal P}$ the polynomial $x_0^{k_\alpha} x^\delta f_\alpha^h$ belongs to the ideal generated by $G$ and, hence, every element in $\langle \id G_{\mathcal P} \rangle$ belongs to $(G)$. 

By Proposition \ref{th:reductionAffineSet}\eqref{rAS_ii}, for every $f \in \Ax$, there is $\bar f \in\langle \mathcal N(\id j)\rangle$ such that $f\xrightarrow{\id G_{\mathcal P}}_\ast\bar f$. Then, the polynomial $\bar f$ is a $\id j$-reduced form of $f$ modulo $(\id G)$ and, by construction, $f-\bar f$ belongs to $\langle \id G_{\mathcal P} \rangle$. Thus, by the previous observation there is a suitable integer $t$ such that $x_0^t(f-\bar f)^h$ belongs to $(G)$. In particular, letting 
\[
t_f=\max\{0, \deg(\bar f)-\deg(f)\},\quad t_{\bar f}=\max\{0, \deg(f)-\deg(\bar f)\},
\]
we have $x_0^{t+t_f}f^{h}-x_0^{t+t_{\bar f}}\bar f^{h}\in (G)$, where $x_0^{t+t_f}\bar f^{h}\in \langle \mathcal N(J_{\geq m})\rangle$ by Corollary \ref{negativo}. Hence, $x_0^{t+t_f} \bar f^{h}$ is the $J_{\geq m}$-normal form of $x_0^{t+t_{\bar f}}f^{h}$ modulo $(G)$, by Theorem \ref{th:varie CMR} \eqref{varie CMR_iii} being $G$ a $\mathcal P(J_{\geq m})$-marked basis. We would obtain a contradiction if $f$ had more than one $\id j$-reduced form modulo $(\id G)$ and we can conclude that $\bar f$ is the $\id j$-normal form of $f$ modulo $(\id G)$. 

Now, it remains to show that the degree of the $\id j$-normal form modulo $(\id G)$ of a polynomial $f$ is bounded from above by $\max\{\deg(f),m\}$ and it is sufficient to argue on the terms of $\Ax$. If a term $x^\beta$ belongs to $\cN(\id j)$, the thesis follows immediately because $x^\beta$ coincides with its $\id j$-normal form.

We now prove the statement for the terms in $\id j$, considering first the 
terms in $\mathcal P(\id j)$, then the terms in $\id j_{\leq m}\setminus \mathcal P(\id j)$, and finally the remaining terms in $\id j\setminus\{\mathcal P(\id j)\cup \id j_{\leq m}\}$.

If $x^\beta$ belongs to $\mathcal P(\id j)$, then $x^\beta\xrightarrow{\id G_{\mathcal P}}_\ast x^\beta-f_\beta$ and $\deg(x^\beta-f_\beta)\leq \max\{\vert \beta\vert,m\}$, by Definition \ref{marked affini}.

Consider now $x^\beta \in \id j_{\leq m}\setminus \mathcal P(\id j)$ and let $H_{\overline \beta}$ be the $J_{\geq m}$-normal form modulo $(G)$ of the term $x_0^{m-\vert\beta\vert}x^\beta$ in $\Aox$. By Proposition \ref{th:reductionAffineSet}  we know that every polynomial in $\Ax$ has a $\id j$-reduced form modulo $(\id G)$.
By the uniqueness of the reduced forms previously proved, $(H_{\overline \beta})^a$ is the $\id j$-normal form of $x^\beta$ modulo $(\id G)$, i.e.~$x^\beta\xrightarrow{\id G_{\mathcal P}}_\ast (H_{\overline \beta})^a$ and by construction we have also $\deg((H_{\overline \beta})^a)\leq \deg(H_{\overline \beta})\leq m$.

Finally, let $E$ be the set of terms $x^\beta$ in $\id j\setminus\{\mathcal P(\id j)\cup \id j_{\leq m}\}$ whose $\id j$-normal form modulo $(\id G)$ has degree exceeding $\vert \beta \vert> m$. Suppose $E$ is non-empty and consider $x^\beta \in E$ with minimal degree $t:=\deg(x^\beta)$ and minimal $x_i:=\min(x^\beta)$ among the monomials of $E$ of degree $t$. Recall that $x^\beta$ belongs to $\id j$ and its degree $t$ is $\geq m+1$.
By Lemma \ref{lemma:importante}\eqref{importante_i}, the term $x^\beta$ can be written as $x^{\beta'} x_i$, with $x^{\beta'} \in \id j_{\leq t-1}$ and $x_i=\min (x^\beta)$. Let $h_{\beta'}$ be the $[\id j,m]$-normal form modulo $(\id G)$ of $x^{\beta'}$, hence $x^{\beta'}-h_{\beta'}\in (\id G)$. By the minimality of $t$ in $E$, we have $h_{\beta'}\in \langle \cN(\id j)_{\leq t-1}\rangle$. For every monomial $x_ix^\eta \in \supp(x_i h_{\beta'})\cap \id j$, we can apply Lemma \ref{lemma:importante}\eqref{importante_ii} because $x^\eta \in  \cN(\id j)_{\leq t-1}$, obtaining that either $x_ix^\eta$ belongs to $\mathcal P(\id j)$ or $x_i>\min(x^\eta)$. If $x_ix^\eta \in \mathcal P(\id j)$, then there is $f\in \id G$ such that $\Ht(f)=x_ix^\eta$ and $x_ix^\eta\xrightarrow{\id G_{\mathcal P}}_\ast T(f)$ with $\deg(T(f))\leq t$ by definition of $\id G$. Otherwise,   $x_ix^\eta=x^{\alpha'}x_\ell$, with $x^{\alpha'}\in \id j$ and $\min(x^\gamma)=x_\ell<x_i$. Then, by the minimality of $x_i$, every monomial of $\supp(x_i h_{\beta'})\cap \id j$ has a $\id j$-reduced form modulo $(\id G)$ of degree $\leq t$. This is a contradiction, so $E$ is empty.
\end{proof}

\begin{theorem}\label{th:markedBasisA2}
Let  $\id G$ be a $[\mathcal P(\id j),m]$-marked set. If every polynomial $f\in \Ax$ has $\id j$-normal form $\Nf(f)$ modulo $(\id G)$ with $\deg(\Nf(f))\leq \max\{\deg(f),m\}$, then $\id G$ is a $[\mathcal P(\id j),m]$-marked basis. 
\end{theorem}

\begin{proof}
We construct a set of marked polynomials whose head terms are the terms in $\id j_{\leq m}\cup \mathcal P(\id j)$. 
For every term $x^\alpha$ in $\mathcal P(\id j)$, we simply consider the marked polynomial $f_\alpha$ in $\id G$ and define $F_\alpha:=x_0^{t_\alpha}f_\alpha^h\in \Aox$, where $t_\alpha:=\max\{m-\deg(f_\alpha),0\}$.

By the hypothesis, for every term $x^\beta \in \id j_{\leq m}\setminus\mathcal P(\id j)$, there is a unique polynomial $\Nf(x^\beta)\in \langle \mathcal N(\id j)\rangle$ such that $x^\beta \xrightarrow{\ \id G_{\mathcal P}\ }_\ast\Nf(x^\beta)$ and $\deg(\Nf(x^\beta))\leq \max\{\vert \beta \vert,m\}=m$. Then, the polynomial $f_\beta=x^\beta-\Nf(x^\beta)$ belongs to $\langle\id G_{\mathcal P}\rangle\subseteq (\id G)$ and $\deg(f_\beta)\leq m$. We define $F_\beta:=x_0^{t_\beta}f_\beta^h\in \Aox$, where $t_\beta:=\max\{m-\deg(f_\beta),0\}$. Note that, by construction, $F_\alpha^a$ and $F_\beta^a$ belong to $(\id G)$.

The following set of polynomials 
\begin{equation}\label{eq:homBasis}
G:=\{F_\alpha\}_{x^\alpha\in \mathcal P(\id j)}\cup\{F_\beta\}_{x^\beta\in \id j_{\leq m}\setminus\mathcal P(\id j)} \subseteq \Aox
\end{equation}
is a $\mathcal P(J_{\geq m})$-marked set, by Proposition \ref{prop:PBtronc}\eqref{prop:PBtronc_i}.
By Theorem \ref{th:varie CMR}\eqref{varie CMR_iii}, the marked set $G$ is a $\mathcal P(J_{\geq m})$-marked basis because,  for every $\ell\geq 0$  and every $F\in (G)_{ \ell}$, $F\xrightarrow{\ G^{ (\ell)}\ }_\ast 0$. Indeed, assume that $F\xrightarrow{\ G^{ (\ell)} \ }_\ast H$. Note that, by Corollary \ref{negativo} and by the construction of $G$, the polynomial $H^a$ is the $\id j$-normal form of $F^a$ modulo $(\id G)$. Since also the null polynomial is the $\id j$-normal form of $F^a$ modulo $(\id G)$, we must have $H^a=0$, hence $H=0$.
\end{proof}

\begin{Corollary}\label{cor:markedBasisA}
Let  $\id G$ be a $[\mathcal P(\id j),m]$-marked set. Then, $\id G$ is a $[\mathcal P(\id j),m]$-marked basis if and only if every polynomial $f\in \Ax$ has $\id j$-normal form $\Nf(f)$ modulo $(\id G)$ with $\deg(\Nf(f))\leq \max\{\deg(f),m\}$. 
\end{Corollary}

\begin{Notation}
Let $\id G$ be a $[\mathcal P(\mathfrak j),m]$-marked basis. By Theorem \ref{th:markedBasisA1}, for every $x^\beta\in \id j_{\leq m}\setminus \mathcal P(\id j)$ we can consider the polynomial $f_\beta:=x^\beta-\Nf(x^\beta)$, where $\Nf(x^\beta)$ belongs to $\langle \mathcal N(\id j)_{\leq m}\rangle$. We denote by $\overline{\id G}$ the set $\{ {f_\beta}\}_{x^\beta\in \id j_{\leq m}\setminus\mathcal P(\id j)}\subseteq (\id G)$.
\end{Notation}

\begin{Corollary}\label{1:1trabasi}
There is a bijective correspondence between the set of $[\mathcal P(\id j),m]$-marked bases and the set of $\mathcal P(J_{\geq m})$-marked bases.
\end{Corollary}

\begin{proof}
If $\id G$ is a $[\mathcal P(\id j),m]$-marked basis, it is sufficient to consider the set $\id G \cup \overline{\id G}$ and homogenize this set of polynomials. Up to multiplication by a suitable power of $x_0$, we obtain the $\mathcal P(J_{\geq m})$-marked basis  $G$ as in \eqref{eq:homBasis}.

Conversely, if $G$ is a $\mathcal P(J_{\geq m})$-marked basis, we immediately obtain a unique $[\mathcal P(\id j),m]$-marked basis following Definition \ref{marked affini}.
\end{proof}

The following result will give a better insight into the notion of $[\mathcal P(\id j),m]$-marked basis.
%, for $\id j$ quasi-stable ideal in $\Ax$.

\begin{Corollary}\label{thm:equivBA}
Let $\id G$ be a $[\mathcal P(\id j),m]$-marked set. The followings are equivalent:
\begin{enumerate}[(i)]
\item \label{thm:equivBA_i} $\id G$ is a $[\mathcal P(\id j),m]$-marked basis;
\item \label{thm:equivBA_ii} $\Ax_{\leq t}=(\id G)_{\leq t} \oplus \langle\mathcal N(\id j)_{\leq t}\rangle$, for every $t\geq m$;
\item \label{thm:equivBA_iii}
$^aH_{\Ax/\id j}(t)= \,^aH_{\Ax/(\id G)}(t)$, for every $t\geq m$. 
\end{enumerate}
\end{Corollary}

\begin{proof}
We assume first that $\id G$ is a $[\mathcal P(\id j),m]$-marked basis and prove that statement \eqref{thm:equivBA_ii} holds. Let $f$ be a polynomial of $\Ax$ and $t:=\max\{m,\deg(f)\}$. By Theorem \ref{th:markedBasisA1}, $f$ has a unique $\id j$-normal form $\Nf(f)$ modulo $(\id G)$ with $\deg(\Nf(f))\leq t$. Hence, $f\in \Ax_{\leq t}$ can be written as $f=(f-\Nf(f))+\Nf(f)$, where $f-\Nf(f)\in (\id G)_{\leq t}$ and $\Nf(f)\in \langle \mathcal N(\id j)_{\leq t}\rangle$. The uniqueness of this writing is a consequence of the uniqueness of normal forms.

Conversely, assume now that $\Ax_{\leq t}=(\id G)_{\leq t} \oplus \langle\mathcal N(\id j)_{\leq t}\rangle$ for every $t\geq m$. Hence, for every $f\in \Ax_{\leq t}$, where $t=\max\{m,\deg(f)\}$, we consider the unique writing $f=f_1+f_2$ with $f_1\in (\id G)_{\leq t}$ and $f_2\in \langle \mathcal N(\id j)_{\leq t}\rangle$. Hence, $f_2$ is a $\id j$-reduced form modulo $\id G$ for the polynomial $f$, and $\deg(f_2)\leq t$. Observe that this reduced form is unique: if $f$  had two $\id j$-reduced forms modulo $(\id G)$, denoted by $g_1$ and $g_2$, then $g_1-g_2$ belongs to $(\id G)_{\leq t'} \cap \langle\mathcal N(\id j)_{\leq t'}\rangle$ for a suitable $t'\geq m$, against the hypothesis. By Theorem \ref{th:markedBasisA2}, $\id G$ is a $[\id j,m]$-marked basis.

The equivalence between \eqref{thm:equivBA_ii} and \eqref{thm:equivBA_iii} is immediate.
\end{proof}

Finally, \hskip 1mm we \hskip 1mm show \hskip 1mm that \hskip 1mm $[\mathcal P(\mathfrak j),m]$-marked bases \hskip 1mm have \hskip 1mm the expected \hskip 1mm good behavior with respect to the homogenization. We let\hskip 1mm $\id G^h:=\{f^h \ \vert \ f\in \id G\}$ and \hskip 1mm$\overline{\id G}^h:=\{f^h \ \vert \ f\in \overline{\id G}\}$.

\begin{theorem}\label{omogbene}
If $\id G$ is a $[\mathcal P(\id j),m]$-marked basis and $G$ is its corresponding $\mathcal P({\id j^h}_{\geq m})$-marked basis (in the sense of Corollary \ref{1:1trabasi}), then $(G)={(\id G)^h}_{\geq m}=(\id G^h\cup \overline{\id G}^h)_{\geq m}$ and $(\id G)^h=(\id G^h\cup \overline{\id G}^h)^\sat=(G)^{\sat}$.
\end{theorem}

\begin{proof}
For every ideal in $\Ax$, we can obtain its homogenization in $\Aox$ considering the ideal generated by the homogenization of a set of generators and saturating it with respect to the ideal $(x_0)$ (for example, see \cite[Corollary 4.3.8]{KR2}). 
Hence, if we consider the set of generators $\id G\cup \overline{\id G}$ of $(\id G)$, then we obtain ${(\id G)^h}=((\id G^h\cup \overline{\id G}^h):x_0^\infty)$. 

On the other hand, the ideal $(G)\subseteq \Aox$ belongs to $\Mf({\mathcal P}({\id j^h}_{\geq m}))$  and $((G):x_0^\infty)=(G)^{\sat}$ by Theorem \ref{bastaaffine}, in particular $((G):x_0^\infty)_{\geq m}={(G)^{\sat}}_{\geq m}=(G)$, because $G$ is $m$-saturated by Corollary \ref{sattronc}. 
By construction, $G$ is contained in $(\id G^h\cup \overline{\id G}^h)$. Thus, the inclusion $((G):x_0^\infty)\subseteq ((\id G^h\cup \overline{\id G}^h):x_0^\infty)$ is obvious. For the other inclusion, let $g$ be a polynomial in $((\id G^h\cup \overline{\id G}^h):x_0^\infty)$. Then, there is an integer $t$ such that $x_0^t g=\sum a_jx^{\delta_j}f_{\alpha_j}^h$, with $f_{\alpha_j}\in \id G\cup \overline{\id G}$, and by construction there is an integer $\overline m$ such that $x_0^{t+\overline m}g$ belongs to $(G)$. So, also the other inclusion holds true and we obtain $(G)^{\sat}=((G):x_0^\infty) =((\id G^h\cup \overline{\id G}^h):x_0^\infty)={(\id G)^h}$. In particular,  we have $(G)={(G)^{\sat}}_{\geq m}=((G):x_0^\infty)_{\geq m} =((\id G^h\cup \overline{\id G}^h):x_0^\infty)_{\geq m}={(\id G)^h}_{\geq m}$.

We have also obtained $(G)_{\geq m}\subseteq (\id G^h\cup \overline{\id G}^h)_{\geq m}$ and $(\id G^h\cup \overline{\id G}^h)_{\geq m}\subseteq {(\id G)^h}_{\geq m}$, hence the equality $(\id G^h\cup \overline{\id G}^h)_{\geq m}=(\id G)^h_{\geq m}$. Saturating both the homogeneous ideals, we get $(\id G)^h=(\id G^h\cup \overline{\id G}^h)^\sat$. 
\end{proof}

\begin{Corollary}
If $\id G$ is a $[\mathcal P(\id j),m]$-marked basis, then the satiety of the ideal $(\id G^h \cup \overline{\id G}^h)$ is bounded from above by $m$. 
\end{Corollary}

Theorem \ref{omogbene} highlights that $[\id j,m]$-marked bases behave well with respect to homogenization, similarly to Macaulay bases, for the standard grading \cite[Theorem 4.3.19]{KR2}. However, the $[\id j,m]$-marked basis $\id G$, jointly with $\overline{\id G}$, is not a Macaulay basis for the ideal $(\id G)$, because the set $\id G^h\cup \overline{\id G}^h$ in general does not generate $(\id G)^h$, a saturation is needed, as shown by the following example.

%The following example shows that it is not true that, if $\id G$ is a $[\mathcal P(\id j),m]$-marked basis, then the set $\id G^h \cup \overline{\id G}^h$ generates $(\id G)^h$ without saturating.

\begin{Example}
Consider $\id j=(x_2,x_1^4)$, $m=3$ and $\id G=\{x_2-x_1^3+x_1^2+x_1+2,x_1^4-2x_1^3-2-x_1\}\subset K[x_1,x_2]$ which is a $[\mathcal P(\id j),m]$-marked basis. By $\SGred$, we can compute the $\id j$-reduced forms modulo $(\id G)$ of the monomials in $\id j_{\leq m}\setminus \mathcal P(\id j)$ and construct in this way the following polynomials:
\[\begin{array}{l}
f_1:=x_2^3-9x_1^3-10-7x_1+21x_1^2, \quad f_2:= x_2^2x_1-x_1^3+6+x_1-3x_1^2,\\ 
f_3:= x_2^2+3x_1^3-2-3x_1-7x_1^2, \quad
f_4:=x_2x_1^2-x_1^3-2-3x_1+x_1^2, \\ f_5:=x_2x_1-x_1^3-2+x_1+x_1^2, 
\end{array}
\]
obtaining $\overline{\id G}=\{ f_1, f_2, f_3, f_4, f_5\}$. In this case, ${(\id G)^h}_{\geq 3}$ is equal to $(\id G^h\cup \overline{\id G}^h)_{\geq 3}$, but the equality does not hold if we consider the truncation of both ideals from degree $2$ on. Indeed, the polynomial $3 f_5+ f_3$ has degree $2$, its homogenization belongs to $(\id G)^h$ by definition, however it does not belong to $(\id G^h\cup \overline{\id G}^h)$. 
\end{Example}

\section{Effective criterion for $[\mathcal P(\id j),m]$-marked bases}\label{sectionJmBasis}

Also in this section, $\id j$ is a quasi-stable ideal in $\Ax$ and $m$ is a non-negative integer.

The following theorem will give an algorithmic criterion to test if a $[\mathcal P(\id j),m]$-marked set of polynomials is a $[\mathcal P(\id j),m]$-marked basis (see Section \ref{algoritmi}). 
Furthermore, this theorem will imply interesting theo\-re\-ti\-cal results which improve the understanding of the structure of marked bases and families (e.g. Corollaries \ref{sollsiz} and  \ref{cor:ro2}, Proposition \ref{prop:immersione}, Theorem \ref{th:flatness}).

\begin{theorem}\label{critsupermin} 
A $[\mathcal P(\id j),m]$-marked set $\id G$ is a $[\mathcal P(\id j),m]$-marked basis if and only if the two following statements hold:
\begin{enumerate}[(i)]
\item \label{critsupermin_i} for every $x^{\beta} \in {\id j}_{\leq m}\setminus \mathcal P(\id j)$, $x^{\beta}\SGred_\ast h_{\beta}$ with $h_{\beta}\in\langle\cN(\id j)_{\leq m}\rangle$;
\item \label{critsupermin_ii} for every $f_\alpha \in {\id G}$ and $x_i > \min(x^\alpha)$, \ $x_if_\alpha  \xrightarrow{\ \id G_{\mathcal P}\ }_\ast 0$.
\end{enumerate}
\end{theorem}

\begin{proof} 
If $\id G$ is a $[\mathcal P(\id j),m]$-marked basis, both items \eqref{critsupermin_i} and \eqref{critsupermin_ii} hold by Theorem \ref{th:markedBasisA1}. Indeed, by the uniqueness of $\id j$-reduced forms modulo $(\id G)$, for every $f\in (\id G)$, we have $\Nf(f)=0$ and this implies \eqref{critsupermin_ii}. Furthermore, the bound on the degree of the $\id j$-normal forms modulo $(\id G)$ implies \eqref{critsupermin_i}.

We now prove that if \eqref{critsupermin_i} and \eqref{critsupermin_ii} hold, then we have $\Ax_{\leq t}=(\id G)_{\leq t}\oplus \langle \mathcal N(\id j)_{\leq t}\rangle$ for every $t\geq m$ and then apply Corollary \ref{thm:equivBA}.

We first prove that, if item \eqref{critsupermin_i} holds, then $\Ax_{\leq t}=\langle \id G_{\mathcal P}\rangle_{\leq t}\oplus \langle \mathcal N(\id j)_{\leq t}\rangle$ for every $t\geq m$. By Proposition \ref{th:reductionAffineSet}, we have $\Ax =\langle \id G_{\mathcal P}\rangle \oplus \langle \mathcal N(\id j) \rangle$, because $\id G$ is a $[\mathcal P(\id j),m]$-marked set. It is enough to consider a term $x^\beta$ in $\Ax_{\leq t}$ with $t=\max\{m,\vert \beta\vert\}$, and to prove that the unique writing $x^\beta=g_\beta+h_\beta$ has $g_\beta \in\langle \id G_{\mathcal P}\rangle_{\leq t}$ and $h_\beta \in \langle \mathcal N(\id j)_{\leq t}\rangle$.

If $x^\beta\in \mathcal N(\id j)$, we take the writing $x^\beta=0+x^\beta$. If $x^\beta \in \mathcal P(\id j)$, then there is $f_\beta \in \id G$ such that $\Ht(f_\beta)=x^\beta$ and we have the writing $x^\beta=f_\beta+(x^\beta-f^\beta)$ and conclude observing that $x^\beta-f_\beta\in  \langle\mathcal N(\id j)_{\leq t}\rangle$  by the Definition \ref{marked affini}. If $x^\beta \in \id j_{\leq m}\setminus \mathcal P(\id j)$, by item \eqref{critsupermin_i} we can consider $x^{\beta}\SGred_\ast h_{\beta} \in \langle\mathcal N(\id j)_{\leq m}\rangle$ and obtain the writing $x^\beta=g_\beta+h_\beta$, where $g_\beta$ is the polynomial computed by the reduction process $\xrightarrow{\ \id G_{\mathcal P}\ }$ which belongs to  $\langle \id G_{\mathcal P}\rangle_{\leq t}$. For every other term in $\Ax$, we can repeat the arguments used in the final part of the proof of Theorem \ref{th:markedBasisA1}.

We conclude proving that item \eqref{critsupermin_ii} implies that $\langle \id G_{\mathcal P}\rangle=(\id G)$. We only prove the non-obvious inclusion. In particular, we prove that $x^\delta f_\alpha$ belongs to $\langle \id G_{\mathcal P}\rangle$, for every $f_\alpha \in \id G$ and for every $x^\delta$. 

If $x^\delta=x_i$ for some $i$, item \eqref{critsupermin_ii} implies that $x_if_\alpha$ belongs to   $\langle \id G_{\mathcal P}\rangle$. 

Assume now  that $\vert \delta \vert>1$, and by inductive hypothesis that $x^{\delta'} f_\alpha$ belongs to $\langle \id G_{\mathcal P}\rangle$ for every $x^{\delta'}<_{lex}x^{\delta}$.
We only consider the non-trivial case where $x_i:=\max(x^\delta)>\min(x^\alpha)$. We define $x^{\delta'}:=x^\delta/x_i<_{lex}x^\delta$ and by inductive hypothesis we can consider the writing $x^{\delta'}f_\alpha=\sum \lambda_{\gamma'\alpha'}x^{\gamma'}f_{\alpha'}$, with $x^{\gamma'}f_{\alpha'}\in\id G_{\mathcal P}$,  $\lambda_{\gamma'\alpha'}\in A$ and observe that $x^{\gamma'}<_{lex}x^{\delta'}$, by Lemma \ref{lemma:importante} \eqref{importante_iii}. Then we consider $x^\delta f_\alpha=\sum  \lambda_{\gamma'\alpha'}x_ix^{\gamma'}f_{\alpha'}$. If $x_i<\min(x^{\alpha'})$, then $x_ix^{\gamma'}f_{\alpha'}$ belongs to $\id G_{\mathcal P}$, otherwise it is sufficient to observe that $x_ix^{\gamma'}<_{lex}x^\delta$, hence by inductive hypothesis we have that $x_ix^{\gamma'}f_{\alpha'}\in \langle \id G_{\mathcal P}\rangle$.
\end{proof}

The following result highlights another feature of $[\mathcal P(\id j),m]$-marked bases that involves the notion of syzygy and is analogous to a property of Gr\"obner bases (see \cite{Artin,BaMu} for the role of this property in the study of flatness).

\begin{Corollary}\label{sollsiz}
If $\id G$ is a $[\mathcal P(\id j),m]$-marked basis, then every homogeneous syzygy of $\id j$ lifts to a syzygy of  $\id G$.
\end{Corollary}

\begin{proof}
By results of \cite[Section 5]{Seiler2009II}, a set of generators of the module of the first syzygies of $\id j$ is given by the pairs $(x_i,x^\delta)$ such that $x_ix^{\alpha_1}-x^\delta x^{\alpha_2}=0$, where $x^{\alpha_1},x^{\alpha_2}\in \mathcal P(\id j)$, $x_i>\max(x^{\alpha_1})$ and $x^\delta x^{\alpha_2}\in \mathcal C_{\mathcal P}(x^{\alpha_2})$.
Since $\id G$ is a $[\mathcal P(\id j),m]$-marked basis, for every $f_\alpha\in \id G$ and every $x_i>\min(x^\alpha)$, we have $x_i f_\alpha\SGred 0$ by Theorem \ref{critsupermin} \eqref{critsupermin_ii}. This means that
$x_i f_\alpha=\sum \lambda_ix^{\delta_i} f_{\alpha_i}$, where $x^{\delta_i}f_{\alpha_i}\in \mathfrak G_{\mathcal P}$. In particular, among the polynomials  $x^{\delta_i} f_{\alpha_i}$ there is $x^{\delta'} f_{\alpha'}$ such that $x_ix^\alpha\in \mathcal C_{\mathcal P}(x^{\alpha'})$.
\end{proof}

In the following, we will focus our attention on two issues that improve the application of Theorem  \ref{critsupermin}. %First, we will show that under a suitable hypothesis on $m$ a $[\mathcal P(\id j),m]$-marked basis is actually a $[\mathcal P(\id j),m-1]$-marked basis. 
% Secondly, we will consider the case of $0$-dimensional ideals.

\subsection{Minimal $m$ for a $[\mathcal P(\id j),m]$-marked basis} \label{conseguenze1}
Let $\rho$ be the satiety of the quasi-stable ideal $\mathfrak j\subseteq \Ax$. By Lemma \ref{lemma:importante} \eqref{importante_0} $\rho$ is the maximal degree of the terms in $\mathcal P(\id j)$ divisible by $x_1$. 

First of all, we show that, if $m\geq \rho$ and $\id G$ is a $[\mathcal P(\id j),m]$-marked basis, then the degree of the $[\mathcal P(\id j),m]$-normal form of every term in $\Ax$ is bounded from above by the maximum between $m-1$ and the degree of the term itself. 

\begin{Lemma}\label{ro}
Let $m\geq \rho$. If $\id G$ is a $[\mathcal P(\id j),m]$-marked basis, then for every term $x^\beta \in \Ax$ we have $\Nf(x^\beta)\in \langle \cN(\id j)_{\leq t}\rangle$ where $t=\max\{m-1,\vert\beta\vert\}$.
\end{Lemma}

\begin{proof}
By Theorem \ref{th:markedBasisA1}, %since $\id G$ is a $[\id j,m]$-marked basis,
$[\id j,m]$-reduced forms modulo $(\id G)$ are unique and for every $x^\beta \in \Ax$ the degree of $\Nf(x^\beta)$ is bounded from above by $\max\{m,\vert \beta \vert\}$. Furthermore, the $\id j$-normal forms modulo $(\id G)$ can be computed by $\SGred$ (Proposition \ref{th:reductionAffineSet}).  Then, for every $x^\beta \in \Ax$, we have $x^\beta\SGred \Nf(x^\beta)\in \langle \mathcal N(\id j)_{\leq t}\rangle$, where $t=\max\{m,\vert \beta \vert\}$. 
We now assume that $m\geq \rho$ and prove that for every $x^\beta \in \Ax$, the degree of $\Nf(x^\beta)$ is bounded from above by $\max\{m-1,\vert \beta \vert\}$. 

If $x^\beta \in \mathcal N(\id j)$, then $\Nf(x^\beta)=x^\beta \in \cN(\id j)_{\leq \vert\beta\vert}$. If $x^\beta\in \id j$ and $\vert\beta\vert\geq m$, then $t=\vert\beta\vert$ and, by definition of $[\id j,m]$-reduced form modulo $(\id G)$, we have $\supp(\Nf(x^\beta))\subseteq \cN(\id j)_{\leq \vert\beta\vert}$. 
If $x^\beta \in {\id j}$ with $\vert\beta\vert\leq m-1$, then its $[\id j,m]$-normal form has degree lower than or equal to $m$.
By Proposition \ref{th:reductionAffineSet} \eqref{rAS_iii}, there is a unique writing $x^\beta=\sum \lambda_{\delta\alpha}x^\delta f_\alpha +h_\beta$, where $x^\delta f_\alpha\in\id  G_{\mathcal P}$, $\lambda_{\delta\alpha}\in A$ and $h_\beta=\Nf(x^\beta)$, being $\id G$ a $[\mathcal P(\id j),m]$-marked basis. We consider the equality $x_1x^\beta=\sum \lambda_{\delta\alpha} x_1x^\delta f_\alpha +x_1\Nf(x^\beta)$ and observe that $\sum \lambda_{\delta\alpha}x_1x^\delta f_\alpha$ belongs to $\langle\id G_{\mathcal P}\rangle$, since $x_1\leq \min(x^\alpha)$ for every $x^\alpha\in \mathcal P(\id j)$. 

We can obtain $\Nf(x_1x^\beta)$ by further reducing $x_1\Nf(x^\beta)$. 
We focus on the terms $x_1x^\epsilon$ that appear in $x_1\Nf(x^\beta)$ with $\vert\epsilon \vert=m$.  Since the degree of $\Nf(x_1x^\beta)$ is bounded from above by $m$ by Theorem  \ref{th:markedBasisA1}, then $x_1x^\epsilon$ does not appear in $\Nf(x_1x^\beta)$, because its degree is $m+1$. Hence, $x_1x^\epsilon$ can be reduced by $\SGred$, in other words $x_1x^\epsilon$ belongs to $\id j$. By Lemma \ref{lemma:importante} \eqref{importante_ii}, we should have $x_1x^\epsilon\in \mathcal P(\id j)$, but this contradicts the hypothesis $m\geq \rho$.   As a consequence, there are no terms in $\Nf(x^\beta)$ whose degree exceeds $m-1$.
\end{proof}

\begin{theorem}\label{ro2}
A $[\mathcal P(\id j),m]$-marked set $\id G$, with $m \geq \rho$, is a $[\mathcal P(\id j),m-1]$-marked basis if and only if $\id G$ is a $[\mathcal P(\id j),m]$-marked basis. 
\end{theorem}

\begin{proof}
If $\id G$ is a $[\mathcal P(\id j),m-1]$-marked basis, then $\Ax_{\leq t}=(\id G)_{\leq t}\oplus \langle \cN(\id j)_{\leq t}\rangle $ for every $t\geq m-1$ by Corollary \ref{thm:equivBA}. Then  $\id G$ is also a $[\mathcal P(\id j),m]$-marked basis.

Conversely, let $\id G$ be a $[\mathcal P(\id j),m]$-marked basis. Then $\Ax_{\leq t}=(\id G)_{\leq t}\oplus \langle \cN(\id j)_{\leq t}\rangle$ for every $t\geq m$ and it is sufficient to prove that $\id G$ is a $[\mathcal P(\id j),m-1]$-marked set and that $\Ax_{\leq m-1}=\langle \cN(\id j)_{\leq m-1}\rangle\oplus (\id G)_{\leq m-1}$. Then, we conclude by Corollary \ref{thm:equivBA}.

By Theorem \ref{th:markedBasisA1} and Lemma \ref{ro}, for every $x^\alpha \in \mathcal P(\id j)$ with  $\vert\alpha\vert\leq m-1$,  the marked polynomial $f_\alpha=x^\alpha-\Nf(x^\alpha)$  belonging to $\id G$ has degree $\leq m-1$.  Hence, $\id G$ is a $[\mathcal P(\id j),m-1]$-marked set.

We now prove that for every $x^\gamma \in \Ax$ with  $\vert\gamma\vert\leq m-1$, there is a unique decomposition $x^\gamma=g_1+ g_2$ with $ g_1\in (\id G)_{\leq m-1}$ and $g_2 \in \langle \cN(\id j)_{\leq m-1}\rangle$.

Let $x^\gamma \in {\id j}$ (the other case being trivial) and $\vert \gamma \vert \leq m-1$. Since $\id G$ is a $[\id j,m]$-marked basis, we can compute $\Nf(x^\gamma)$ and see by Lemma  \ref{ro} that its degree is $\leq m-1$.

Therefore, $x^\gamma= (x^\gamma-\Nf(x^\gamma) )+\Nf(x^\gamma)\in (\id G)_{\leq m-1} + \langle \cN(\id j)_{\leq m-1}\rangle$. Moreover, the sum is direct as $(\id G)_{\leq m-1} + \langle \cN(\id j)_{\leq m-1}\rangle$ is contained in $(\id G)_{\leq m} \oplus  \langle \cN(\id j)_{\leq m}\rangle$.
 \end{proof}

Under the hypothesis that $\mathrm{char}(K)=0$, Theorem \ref{ro2} is analogous to \cite[Theorem 5.7]{BCLR} for strongly stable ideals, but the proof given here for a $[\mathcal P(\id j),m]$-marked basis is much easier and gives a better insight into the algebraic structure of marked bases. 

\begin{Corollary}\label{cor:ro2} Let $\id G$ be a $[\mathcal P(\id j),m]$-marked set with $m\geq \rho$.
\begin{itemize}
\item $\id G$ is a $[\mathcal P(\id j),m]$-marked basis if and only if it is a $[\mathcal P(\id j),\rho-1]$-marked basis.
\item  $\id G$ is a $[\mathcal P(\id j), \rho-1]$-marked basis if and only if 
\begin{enumerate}[(i)]
\item for all $x^\beta \in {\id j}_{\leq \rho-1}\setminus \mathcal P(\id j)$, $x^\beta\SGred  h_\beta$ with $h_\beta\in\langle \cN(\id j)_{\leq \rho-1}\rangle$;
\item for every $f_\alpha \in \id G$, for every $x_i>\min(x^\alpha)$, $x_if_\alpha\SGred 0$.
\end{enumerate}
\end{itemize}
\end{Corollary}

\subsection{The 0-dimensional case}\label{conseguenze2}

We now turn to the special case of a $[\id j,m]$-marked set $\id G$ with $K[\bf x]/\id j$ Artinian.  The algorithmic techniques that we can use to check if $\id G$ is a $[\mathcal P(\id j),m]$-marked basis are more efficient in this special case.
Indeed, if $K[\bf x]/\id j$ is Artinian, then $\cN(\id j)$ is a finite set of terms and  we have that $\cN(\id j)_{\leq\reg(\id j)-1}=\cN(\id j)_{\leq\reg(\id j)-1+r}$ for every positive integer $r$. Recall that $\reg(\id j)=\max\{\vert\alpha\vert:x^\alpha\in\mathcal P(\id j)\}$ for a quasi-stable ideal $\id j$.

%\begin{Proposition}\label{punti}
%Let $\id G$ be a $[\mathcal P(\id j),m]$-marked set, with $m\geq \reg(\id j)-1$. If $\cN(\id j)$ is a finite set, then for every $x^\beta \in \id j$, we have $x^\beta\SGred_\ast h_\beta$ with $h_\beta\in \langle \cN(\id j)_{\leq \reg(\id j)-1}\rangle$.
%\end{Proposition}
%
%\begin{proof}
%For every $x^\beta\in \id j$, we compute $x^\beta\SGred_\ast h_\beta$, $h_\beta\in \langle\cN(\id j)\rangle$. We write this polynomial as $h_\beta=h_{1}+h_{2}$, where $h_{1}\in\langle \cN(\id j)_{\leq m}\rangle$ and $h_{2}\in\langle \cN(\id j)_{\leq t}\setminus\cN(\id j)_{\leq m}\rangle$ for some $t\geq m+1$.  Then, we have $h_{2}=0$ and $h_{\beta}\in \langle\cN(\id j)_{\leq \reg(\id j)-1}\rangle$.
%\end{proof}

\begin{Proposition}\label{critpunti}
Let $\id G$ be a $[\mathcal P(\id j),m]$-marked set, with $m\geq \rho$. If $\cN(\id j)$ is finite, then $\id G$ is a $[\mathcal P(\id j),\rho-1]$-marked basis if and only if $x_if_\alpha\SGred 0$, for every $f_\alpha \in \id G$ and for every $x_i> \min(x^\alpha)$. 
\end{Proposition}

\begin{proof}
Since $\cN(\id j)$ is finite, we have that $\rho=\reg(\id j)$ \cite[Lemma (1.7)]{BS}. Then for every $x^\beta  \in \id j_{\leq \rho-1}\setminus B_{\id j}$, if $x^\beta \SGred_\ast h_\beta$, then $h_\beta$ always belongs to $\langle \cN(\id j)_{\rho-1}\rangle$, since $\cN(\id j)$ is finite. Hence, condition \eqref{critsupermin_i} of Theorem \ref{critsupermin} need not to be checked.
\end{proof}

\section{Marked families and flatness}\label{sec:funt}

In the present section, we investigate the marked family $\Mf({\mathcal P}(\id j),m)$ for a given quasi-stable ideal $\id j\subseteq \Kx$ and an integer $m\geq 0$. For every  $K$-algebra $A$, in the following we will denote by $\id j$ also the ideal $\id j\cdot \Ax$ and by $\langle \mathcal N(\id j)_{\leq t}\rangle_A$ the $A$-module generated by $ \mathcal N(\id j)_{\leq t}$. Recall that  we denote by $J$ the saturated ideal $\id j^h$.

Our aim consists in proving that $\Mf({\mathcal P}(\id j),m)$ is endowed with an intrinsic scheme structure because it can be obtained as the scheme representing a functor.  In \cite{BLR,LR2} there is an analogous study for homogeneous marked bases over a \emph{strongly stable} ideal which inspires the study described here.

First of all we prove  that  the definition of marked basis  is natural, namely  marked bases are preserved by  the extension of scalars.

\begin{Lemma}\label{lem:costrBase}  In the above setting, let $I$ be an ideal in $\Ax$.
\begin{enumerate}
\item[(i)] 
$I$ is generated by a $[\mathcal P(\id j),m]$-marked basis if and only if  $\mathcal N(\id j)_{\leq t}$ is a basis for the $A$-module $A[\mathbf x]_{\leq t}/I_{\leq t}$ for every $t\geq m$. 
\item[(ii)] If $I$ is generated by a $[\mathcal P(\id j),m]$-marked basis $\id G$ and $\phi \colon A \rightarrow B$ is a $K$-algebra morphism, then $B[\mathbf x]_{\leq t }= ( I\otimes_A  B)_{\leq t}\oplus \langle \mathcal N(\id j)_{\leq t}\rangle_B$  for every $t\geq m$ and  $I\otimes _A B$ is generated by the $[\mathcal P(\id j),m]$-marked basis $\phi(\id G)$. 
\end{enumerate}
\end{Lemma}

\begin{proof} {\textit{(i)}} If  $I$ is generated by a $[\mathcal P(\id j),m]$-marked basis and $t\geq m$, then we conclude by Corollary \ref{thm:equivBA}.  

 For the converse, we construct a $[\mathcal P(\id j),m]$-marked set $\id G\subseteq I$. For every $x^\alpha\in \mathcal P(\id j)$, we denote by $t_\alpha$ the integer $\max\{\vert\alpha\vert,m\}$. By hypothesis, there are  $f_\alpha\in I_{\leq t_\alpha}$ and $h_\alpha\in \langle \mathcal N(\id j)_{\leq t_\alpha}\rangle$ such that $x^\alpha=f_\alpha+h_\alpha$. We  now prove that  the $[\mathcal P(\id j),m]$-marked set $\id G:=\{f_\alpha\}_{x^\alpha \in \mathcal P(\id j)}\subseteq A[\mathbf x]$  generates $I$. For every $f\in I$, by Proposition \ref{th:reductionAffineSet}, we consider the unique polynomial $h\in\langle \mathcal  N(\id j)\rangle$ such that $f\xrightarrow{\ \id G_{\mathcal P} \ }_\ast h$. Hence, for $t= \max\{\deg(f),\deg(h),m\}$, we have the writing  $f=(f-h)+h$ in $\langle \id G_{\mathcal P}\rangle_{\leq t}\oplus\langle\mathcal N(\id j)_{\leq t}\rangle$  (see Proposition \ref{th:reductionAffineSet} \eqref{rAS_iii}). Since $\langle \id G_{\mathcal P}\rangle_{\leq t}\subseteq I_{\leq t}$, then by hypothesis on $I$, $h=0$ and $f\in \langle \id G_{\mathcal P}\rangle_{\leq t}\subseteq (\id G)$.
By Corollary \ref{thm:equivBA}, we conclude that  $I$ is generated by the $[\mathcal P(\id j),m]$-marked basis $\id G$.

{\textit{(ii)}}  By hypothesis  $I$ is generated by a $[\mathcal P(\id j),m]$-marked basis  $\id G$, so that $I\otimes_AB$ is generated by the  $[\mathcal P(\id j),m]$-marked set  $\phi(\id G) $. 

Furthermore, we have $A[\mathbf x] =  I  \oplus \langle \mathcal N(\id j)\rangle_A$ and, by Corollary \ref{thm:equivBA},   $\Ax_{\leq t}=I_{\leq t } \oplus \langle \mathcal N(\id j)_{\leq t}\rangle_A$ for every $t\geq m$.

Recalling that the tensor product commutes with the direct sum  we obtain  both $B[\mathbf x]= (I \otimes_AB) \oplus \langle \mathcal N(\id j)\rangle_B $ and,  for every $t\geq m$, $B[\mathbf x]_{\leq t }= ( I_{\leq t}\otimes_A  B)\oplus \langle \mathcal N(\id j)_{\leq t}\rangle_B$; as a consequence   the inclusion of $B$-modules  $(I\otimes_AB)_{\leq t}\supseteq  I_{\leq t}\otimes_A  B$   turns ut to be an equality for every $t\geq m$. Therefore,   
$B[\mathbf x]_{\leq t }= (I\otimes_AB)_{\leq t}\oplus \langle \mathcal N(\id j)_{\leq t}\rangle_B$ and we conclude by {\textit{(i)}} that $\phi(\id G) $ is in fact a  $[\mathcal P(\id j),m]$-marked basis. 
\end{proof}

\begin{Definition}\label{def:functor}
We define the following functor between the category of Noetherian $K$-algebras and that of sets:
\[
\underline{\mathrm{Mf}}_{\mathcal P(\id j),m}:  \underline{\text{Noeth-}K\text{-Alg}}\rightarrow \underline{\mathrm{Set}}
\]
which associates to every Noetherian $K$-algebra $A$ the set
\[\underline{\mathrm{Mf}}_{\mathcal P(\id j),m} (A) :=\{I\subseteq A[\mathbf x]\ \vert\ A[\mathbf x]_{\leq t}=I_{\leq t}\oplus \langle \mathcal N(\id j)_{\leq t}\rangle_A, \ \forall t\geq m\}.
\]
and to every $K$-algebra morphism  $\phi : A\rightarrow B$ the function  
$$\overline{\phi} \colon \underline{\mathrm{Mf}}_{\mathcal P(\id j),m}(A) \rightarrow \underline{\mathrm{Mf}}_{\mathcal P(\id j),m}(B)  \hbox{\  given by \ } \overline{\phi} (I )= I\otimes_A B. $$ 

Note that $\overline{\phi} (I)$ belongs to $\underline{\mathrm{Mf}}_{\mathcal P(\id j),m}(B) $ by Lemma \ref{lem:costrBase}{\textit{(ii)}}.
\end{Definition}

\begin{Remark}\label{rk:NotaMF}The symbol we have chosen to denote the functor $\underline{\mathrm{Mf}}_{\mathcal P(\id j),m}$ hints at a close relation with $[\mathcal P(\id j), m]$-marked bases.
Indeed,   by the previous  Lemma \ref{lem:costrBase}{\textit{(i)}}, for every Noetherian $K$-algebra $A$, $\underline{\mathrm{Mf}}_{\mathcal P(\id j),m} (A)$ is precisely the   $[\mathcal P(\id j), m]$-marked family in $\Ax$.  
\end{Remark}

We now explicitly construct an ideal $\id U\subset K[C]$ so that the affine scheme $\mathrm{Spec}(K[C]/\id U)$ represents the functor $\underline{\mathrm{Mf}}_{\mathcal P(\id j),m}$.

For every $x^\alpha \in \mathcal P(\id j)$, let $t_\alpha$ be the integer $\max\{m,\vert\alpha\vert\}$. We define the following set of parameters: 
\begin{equation}\label{eq:parameters}
C:=\{C_{\alpha\eta}\ \vert \ x^\alpha\in \mathcal P(\id j), x^\eta \in \mathcal N(\id j)_{\leq t_\alpha}\}.
\end{equation}
For each $x^\alpha \in \mathcal P(\id j)$, we consider the following marked polynomial in $K[C][\mathbf x]$:
\[
f_\alpha:=x^\alpha-\sum_{x^\eta \in \mathcal N(\id j)_{\leq t_\alpha}}C_{\alpha\eta}x^\eta.
\]
We collect these marked polynomials in the following $[\mathcal P(\id j),m]$-marked set.
\begin{equation}\label{eq:MSwithParam}
\id G:=\{f_\alpha\ \vert\ x^\alpha \in \mathcal P(\id j)\}\subseteq K[C][\mathbf x].
\end{equation}

For every $x^\beta \in \id j_{\leq m}\setminus \mathcal P(\id j)$, let $h_\beta$ be the unique polynomial in $\langle \mathcal N(\id j)\rangle_{K[C]}$ such that $x^\beta\xrightarrow{\ \id G_{\mathcal P} \ }_\ast h_\beta$ (see Proposition \ref{th:reductionAffineSet}). We  write 
\begin{equation}\label{eq:decomposition}
h_\beta=h_{\beta}^{\left(\leq m\right)}+h_{\beta}^{(>m)}
\end{equation}
with $h_{\beta}^{(\leq m)}\in \langle \mathcal N(\id j)_{\leq m}\rangle_{K[C]}$ and $h_{\beta}^{(>m)}\in \langle\mathcal N(\id j)_{\leq \deg(h_{\beta})}\setminus \mathcal N(\id j)_{\leq m}\rangle_{K[C]}$.

For every $f_\alpha \in \id G$ and for every $x_i>\min(x^\alpha)$, let $h_{i\alpha}$ be the unique polynomial in $\langle \mathcal N(\id j)\rangle\subseteq K[C][\mathbf x]$ such that $x_if_\alpha\xrightarrow{\ \id G_{\mathcal P} \ }_\ast h_{i\alpha}$ (see Proposition \ref{th:reductionAffineSet}).

\begin{Notation}\label{not:U}
We denote by $\id U\subset K[C]$ the ideal generated by the coefficients in $K[C]$ of the polynomials $h_{\beta}^{(>m)}$, for every $x^\beta \in \id j_{\leq m}\setminus \mathcal P(\id j)$, and by the coefficients of the polynomials $h_{i\alpha}$, for every $f_\alpha \in \id G$ and for every $x_i>\min(x^\alpha)$.
\end{Notation}

\begin{Remark}\label{rem:funt}
The ideal $\id U$ is exactly the ideal constructed by the \emph{su\-per\-mi\-ni\-mal reduction} of homogeneous polynomials in \cite[Proposition 5.9]{BCLR}, under the hypothesis that $J=\id j^h$ is strongly stable. Hence, our construction includes the one in \cite{BCLR}, but it is more general, since we assume the weaker hypothesis that $\id j$ (and hence $\id j^h$) is quasi-stable. 
\end{Remark}

\begin{theorem}\label{thm:schemamarcatoA}
In the above setting, $\underline{\mathrm{Mf}}_{\mathcal P(\id j),m}$ is a representable functor, whose representing scheme is ${\mathrm{Mf}}_{\mathcal P(\id j),m}:=\mathrm{Spec}(K[C]/\id U)$, that we call \emph{$[\mathcal P(\id j),m]$-marked scheme}.
\end{theorem}

\begin{proof}
This result and its proof are analogous to \cite[Theorem 2.6]{LR2} for the case of strongly stable ideals. For the sake of completeness, we propose also here the proof.

Consider the $[\mathcal P(\id j),m]$-marked set $\id G$ in \eqref{eq:MSwithParam}.  By Lemma \ref{lem:costrBase}, for every $K$-algebra $A$, a $[\mathcal P(\id j),m]$-marked set in $A[\mathbf x]$ is uniquely and completely defined by a $K$-algebra morphism $\varphi: K[C]\rightarrow A$, defined by $\varphi(C_{\alpha\gamma})=c_{\alpha\gamma}\in A$ for every $x^\alpha \in \mathcal P(\id j), x^\gamma \in \mathcal N(\id j)_{\max \{\vert\alpha\vert, m\}}$. We extend $\varphi$ to a morphism from $K[C][\mathbf x]$ to $A[\mathbf x]$ in the obvious way.

It is sufficient to observe, by Theorem \ref{critsupermin}, that $\varphi(\id G)\subset A[\mathbf x]$ is a $[\mathcal P(\id j),m]$-marked basis if and only if the generators of $\id U$ vanish at  $c_{\alpha\gamma}\in A$. 

Hence, $\varphi(\id G)$ is a $[\mathcal P(\id j),m]$-marked basis in $A[\mathbf x]$ if and only if $\ker(\varphi)\supseteq \mathfrak U$. In this case, $\varphi$ factors through $K[C]/\id U$. The induced $K$-algebra morphism  from $K[C]/\id U$ to $A$ defines a scheme morphism $\mathrm{Spec}(A) \rightarrow \mathrm{Spec}(K[C]/\id U)$. Therefore, the scheme $\Spec(K[C]/\mathfrak U)$ represents the functor $\underline{\mathrm{Mf}}_{\mathcal P(\id j),m}$.
\end{proof}

\begin{Proposition}\label{prop:rho}
Let $\rho$ be the satiety of $\id j$ and $A$ any Noetherian $K$-algebra.
\begin{itemize}
\item[(i)] If $m\geq\rho$, then $\underline{\mathrm{Mf}}_{\mathcal P(j),m}(A)= \underline{\mathrm{Mf}}_{\mathcal P(j),\rho-1}(A)$; in particular, $\underline{\mathrm{Mf}}_{\mathcal P(j),m}=\underline{\mathrm{Mf}}_{\mathcal P(j),m-1}$.
\item[(ii)] If $m<\rho-1$, then $\underline{\mathrm{Mf}}_{\mathcal P(j),m}(A)\subseteq \underline{\mathrm{Mf}}_{\mathcal P(j),\rho-1}(A)$; in particular, $\underline{\mathrm{Mf}}_{\mathcal P(j),m-1}$ is a closed  subfunctor of $\underline{\mathrm{Mf}}_{\mathcal P(j),m}$.
\end{itemize}
\end{Proposition}

\begin{proof} First, we assume $m\geq \rho$ and
show $\underline{\mathrm{Mf}}_{\mathcal P(j),m}(A)= \underline{\mathrm{Mf}}_{\mathcal P(j),m-1}(A)$ using analogous arguments of \cite[Theorem 3.4]{LR2}. If $I$ is an ideal of $\underline{\mathrm{Mf}}_{\mathcal P(j),m}(A)$, then  
$A[\mathbf x]_{\leq t}=I_{\leq t}\oplus \langle \mathcal N(\id j)_{\leq t}\rangle_A$, for every $t\geq m$, 
and by Lemma \ref{lem:costrBase} $I$ is generated by a $[\mathcal P(\id j),m]$-marked basis $\id G$. By Theorem \ref{ro2}, $\id G$ is also a $[\mathcal P(\id j),m-1]$-marked basis, so by Corollary \ref{thm:equivBA} $I$ belongs to $\underline{\mathrm{Mf}}_{\mathcal P(j),m-1}(A)$. Conversely, it is enough to observe that a $[\mathcal P(\id j),m-1]$-marked basis is always also a $[\mathcal P(\id j),m]$-marked basis. In case $m < \rho-1$ it is sufficient to make the same observation.

It remains to show that that $\underline{\mathrm{Mf}}_{\mathcal P(j),m-1}$ is a closed subfunctor of $\underline{\mathrm{Mf}}_{\mathcal P(j),m}$. We exploit an idea already applied in the proof of \cite[Theorem 3.4]{LR2} that we adapt to our situation. 

For the integer $m$, let $C^{(m)}$ be the set of parameters as in \eqref{eq:parameters}, $\id U^{(m)}$ the ideal as in Notation \ref{not:U} and let $C^{(m-1)}$ and $\id U^{(m-1)}$ the analogous ones for the integer $m-1$. By construction  we have $C^{(m-1)}\subseteq C^{(m)}$. 
Hence we can consider the surjective homomorphism $\Phi: K[C^{(m)}] \rightarrow K[C^{(m-1)}]$ which associates to $C_{\alpha \eta}^{(m)}$ the coefficient of $x^\eta$ in the polynomial $f_\alpha$ when considered for the integer $m-1$. Hence, if $x^\eta$ does not appear in such a polynomial, we associate $0$ to $C_{\alpha \eta}^{(m)}$ and we obtain $\mathfrak U^{(m-1)}=\mathfrak U^{(m)}+ (C_{\alpha \eta}^{(m)}:\Phi(C_{\alpha \eta}^{(m)})=0)$. So, the homomorphism $\Phi$ induces a surjective homomorphism $\bar{\Phi}: K[C^{(m)}]/\id U^{(m)} \rightarrow K[C^{(m-1)}]/\id U^{(m-1)}$ which gives an isomorphism between ${\mathrm{Mf}}_{\mathcal P(j),m-1}$ and a closed subscheme of ${\mathrm{Mf}}_{\mathcal P(j),m}$. 
%Indeed, the epimorphism $ K[C^{(m)}] \rightarrow K[C^{(m-1)}] \rightarrow K[C^{(m-1)}]/\id U^{(m-1)}$ is equal to the homomorphism $K[C^{(m)}] \rightarrow K[C^{(m)}]/\id U^{(m)} \xrightarrow{\bar\Phi} K[C^{(m-1)}]/\id U^{(m-1)}$ and hence $\bar\Phi$ is surjective. 
\end{proof}

In \cite{CMR13} the authors define the representable functor 
\begin{equation}\label{eq:funtore omogeneo}
\underline{\mathrm{Mf}}_{\mathcal P(J_{\geq m})}:  \underline{\text{Noeth-}K\text{-Alg}}\rightarrow \underline{\mathrm{Set}}
\end{equation}
which associates to every Noetherian $K$-algebra $A$ the set
\[\underline{\mathrm{Mf}}_{\mathcal P(J_{\geq m})} (A) :=\{I\text{ homogeneous ideal   in }A[\mathbf x]\  \vert\ A[\mathbf x]=I\oplus\langle \mathcal N(J_{\geq m})\rangle_A\}
\]
and to every $K$-algebra morphism  $\phi : A\rightarrow B$ the function  
$${\phi} \colon \underline{\mathrm{Mf}}_{\mathcal P(J_{\geq m})}(A) \rightarrow \underline{\mathrm{Mf}}_{\mathcal P(J_{\geq m})}(B) \ \hbox{ given by } \  {\phi} (I )= I\otimes_A B.$$ 

We now show that the $[\mathcal P(\id j),m]$-marked scheme  ${\mathrm{Mf}}_{\mathcal P(\id j),m}=\mathrm{Spec}(K[C]/\id U)$ representing $\underline{\mathrm{Mf}}_{\mathcal P(\id j),m}$ represents also the functor $\underline{\mathrm{Mf}}_{\mathcal P(J_{\geq m})}$.

\begin{Proposition}\label{prop:ponte}
For every $K$-algebra $A$, there is a bijective correspondence between the elements in $\underline{\mathrm{Mf}}_{\mathcal P(J_{\geq m})}(A)$ and those in $\underline{\mathrm{Mf}}_{\mathcal P(\id j),m}(A)$.
\end{Proposition}

\begin{proof} It is sufficient to observe that an ideal $\id i\subset A[\mathbf x]$ is generated by a $[\mathcal P(\id j),m]$-marked basis if and only if the ideal $(\id i)^h_{\geq m}\subset A[x_0,\mathbf x]$ is generated by the corresponding $\mathcal P(J_{\geq m})$-marked basis. Indeed, we have the bijective correspondence between $[\mathcal P(\id j),m]$-marked bases in $\Ax$ and $\mathcal P(J_{\geq m})$-marked bases in $\Aox$ described in Corollary \ref{1:1trabasi} and the good behavior of marked bases in $\Ax$ with respect to homogenization investigated by Theorem \ref{omogbene}. Thus, we obtain a bijection between $\underline{\mathrm{Mf}}_{\mathcal P(\id j),m}(A)$ and $\underline{\mathrm{Mf}}_{\mathcal P(J_{\geq m})}(A)$, for every Noetherian $K$-algebra $A$. 
%Then by the Yoneda's Lemma the affine scheme representing $\underline{\mathrm{Mf}}_{\mathcal P(\id j),m}$ 
\end{proof}

\begin{Remark}\label{rem:ident}
Observe that if $I\in \underline{\mathrm{Mf}}_{\mathcal P(J_{\geq m})}(A)$ and $\id i\in \underline{\mathrm{Mf}}_{\mathcal P(\id j),m}(A)$ are matched by the correspondence of Proposition \ref{prop:ponte}, then they define (up to homogenization of $\id i$) the same projective scheme in $\mathbb P^n_A$. 
\end{Remark}

\begin{Corollary}
The affine scheme ${\mathrm{Mf}}_{\mathcal P(\id j),m}$ represents the functor $\underline{\mathrm{Mf}}_{\mathcal P(J_{\geq m})}$ and $\underline{\mathrm{Mf}}_{\mathcal P(\id j),m}\simeq \underline{\mathrm{Mf}}_{\mathcal P(J_{\geq m})}$. 
\end{Corollary}

\begin{proof}
This is an immediate consequence of Proposition \ref{prop:ponte}.
\end{proof}

\begin{Corollary}\label{cor:sottofuntori}
Let $\rho$ be the satiety of $\id j$. If $m\geq \rho$, then $\underline{\mathrm{Mf}}_{\mathcal P(J_{\geq m})}=\underline{\mathrm{Mf}}_{\mathcal P(J_{\geq m-1})}$. If $m<\rho-1$, then $\underline{\mathrm{Mf}}_{\mathcal P(J_{\geq m-1})}$ is a closed  subfunctor of $\underline{\mathrm{Mf}}_{\mathcal P(J_{\geq m})}$. 
\end{Corollary}

\begin{proof}
This is an immediate consequence of Propositions \ref{prop:rho} and \ref{prop:ponte}.
\end{proof}

Referring to the brief but exhaustive description given in \cite[Section 4]{LR2} and the references therein, for any positive integer $n$ and any Hilbert polynomial $p(t)$, we consider the Hilbert functor $\underline{\mathrm{Hilb}}_{p(t)}^n: \underline{\text{Noeth-}K\text{-Alg}}\rightarrow \underline{\mathrm{Set}}$ which associates to any Noetherian $K$-algebra $A$ the set
\begin{equation}\label{eq:defH}
\underline{\mathrm{Hilb}}_{p(t)}^n(A) = \Bigl\{X\subset \mathbb P_A^n: \ X\rightarrow \mathrm{Spec}(A)  \text{ is flat and has fibers with Hilbert polynomial } p(t) \Bigr\}
\end{equation}
and to any $K$-algebra homomorphism $\phi: A \rightarrow B$ the map
$$\begin{array}{lccc}\underline{\mathrm{Hilb}}_n^{p(t)}(\phi): &\underline{\mathrm{Hilb}}_n^{p(t)}(A) &\longrightarrow &\underline{\mathrm{Hilb}}_n^{p(t)}(B)\\
& X &\longmapsto &X \times_{\mathrm{Spec}(A)} \mathrm{Spec}(B).
\end{array}
$$
It is well known that the Hilbert functor is representable and that it can be described as a closed subfunctor of the Grassmannian functor $\underline{\mathrm{Gr}}_{p(r)}^N$ by a natural transformation $\underline{\mathcal H}: \underline{\mathrm{Hilb}}_n^{p(t)}\longrightarrow \underline{\mathrm{Gr}}_{p(r)}^N$, where $r$ is the Gotzmann number of $p(t)$ and $N=\binom{n+r}{n}$. The Grassmannian functor has a well-known cover by open subfunctors \cite[Section III.2.7 and Exercise VI - 18]{EH} given by the Pl\"ucker embedding, which in our framework are defined in the following way. For any set $\mathcal N$ of $p(r)$ distinct monomials of $K[x_0,\mathbf x]_r$, let $\mathfrak J$ be the monomial ideal generating by the monomials of $K[x_0,\mathbf x]_r$ outside $\mathcal N$. The open subfunctor $\underline{\mathrm{G}}_{\mathcal N}$ of $\underline{\mathrm{Gr}}_{p(r)}^N$ associates to every Noetherian $K$-algebra $A$ the set
$$\underline{\mathrm{G}}_{\mathcal N}(A):=\{ A\text{-submodules } L\subseteq A[x_0,\mathbf x]_r \text{ such that } \mathcal N \text{ generates } A[x_0,\mathbf x]_r/L \}.
$$
Using the terminology we have already introduced, if the monomials    of $K[x_0, \mathbf x]_r\setminus \mathcal N$ form the Pommaret basis of a quasi-stable monomial ideal $\mathfrak J$, then 
$$\underline{\mathrm{G}}_{\mathcal N}(A)=\{ A\text{-submodules } L\subseteq A[x_0,\mathbf x]_r \text{ generated by a $\mathcal{P}(\mathfrak J)$-marked set } \}
$$
and 
$$\underline{\mathrm{H}}_{\mathcal N}(A)=\{  L\in \underline{\mathrm{G}}_{\mathcal N}(A) \text{ and  $p(t)$ is  the Hilbert polynomial of  } A[x_0,\mathbf x]/(L)  \}.
$$
By means of the natural transformation $\underline{\mathcal H}$, the open cover given by the subfunctors $\underline{\mathrm{G}}_{\mathcal N}$ induces an open cover of open subfunctors  $\underline{\mathrm{H}}_{\mathcal N}$  of the Hilbert functor.

In \cite{LR2} a relation between $\underline{\mathrm{H}}_{\mathcal N}(A)$ and $\underline{\mathrm{Mf}}_{\mathcal P(\mathfrak J)_{\geq r}}$ is studied when $\mathfrak J$ is a saturated strongly stable ideal \cite[Lemma 4.1, Corollaries 4.2 and 4.3]{LR2}. We now study the analogous relation on the more general assumption that $\mathfrak J$ is a saturated quasi-stable ideal.  

\begin{Lemma}\label{lemma:immersione}
Let $p(t)$ be a Hilbert polynomial with Gotzmann number $r$ and $\mathfrak J$ be a saturated quasi-stable ideal. Then, for every Noetherian $K$-algebra $A$
$${\underline{\mathrm H}}_{\mathcal N(\mathfrak J)_r}(A)\not=\emptyset \Leftrightarrow \text{ the Hilbert polynomial of } A[x_0,\mathbf x]/\mathfrak J \text{ is } p(t).
$$
and, if ${\underline{\mathrm H}}_{\mathcal N(\mathfrak J)_r}(A)\not=\emptyset $, then
$$\underline{\mathrm{H}}_{\mathcal N(\mathfrak J)_r}(A)=\underline{\mathrm{Mf}}_{\mathcal P(\mathfrak J)_{\geq r}} (A).
$$
%where $I_X$ is the saturated defining ideal of $X$.
\end{Lemma}

\begin{proof}
This result and its proof are analogous to \cite[Lemma 4.1]{LR2} for the case of strongly stable ideals.

First of all we observe that equality is a local property; therefore   
without loss of generality we
%it is not limiting to 
assume that $A$ is  local. 

If the Hilbert polynomial of $A[x_0,\mathbf x]/\mathfrak J$ is $p(t)$ then $\mathrm{Proj}(A[x_0,\mathbf x]/\mathfrak J)$ belongs to ${\underline{\mathrm H}}_{\mathcal N(\mathfrak J)_r}(A)$.
Conversely, assume that $X$ is a scheme in ${\underline{\mathrm H}}_{\mathcal N(\mathfrak J)_r}(A)$, let  $I_X$ be the saturated defining ideal of $X$, and let $I:=(I_X)_{\geq r}$. Then, by definition of $\mathcal P(\mathfrak J_{\geq r})$-marked set and by \cite[Theorem 5.9]{CMR13}, for every $m\geq r$, $I_m$ has a free direct summand $P_m$ with rank as an $A$-module equal to that of $\mathfrak J_m$. Then, for every $m\geq r$ the rank of  $A[x_0,\mathbf x]_m/\mathfrak J_m$ cannot be smaller  than that of  $A[x_0,\mathbf x]_m/I_m$, namely than the value $p(m)$ of the Hilbert polynomial $p(t)$ at $m$,  because $r$ is the Gotzmann number of $p(t)$. On the other hand, this rank  cannot be larger than $p(m)$ by Macaulay's Estimate on the Growth of Ideals (see for example \cite[Theorem 3.3]{Gr}).   
Then, the Hilbert polynomial of $ \mathfrak J_m$ is $p(t)$.

Hence, in every degree $m\geq r$, the ideal $I_m$ coincides with the free $A$-module $P_m$, again by  \cite[Theorem 5.9]{CMR13} we obtain $I_m\oplus \langle \mathcal N(\mathfrak J)\rangle_m=P_m\oplus \langle \mathcal N(\mathfrak J)\rangle_m=A[x_0, \mathbf x]_m$. Therefore, by definition  $(I_X)_{\geq r} \in \underline{\mathrm{Mf}}_{\mathcal P(\mathfrak J_{\geq r})} $.

Conversely, if $I\in \underline{\mathrm{Mf}}_{\mathcal P(\mathfrak J_{\geq r})}(A)$, then $I$ is generated by a $\mathcal P(\mathfrak J_{\geq r})$-marked set and its Hilbert polynomial is $p(t)$, so that the scheme $X$ defined by $I$ belongs to $\underline{\mathrm{H}}_{\mathcal N(\mathfrak J)_r}(A)$.
\end{proof}

\begin{Proposition}\label{prop:immersione}
Let $\rho$ be the satiety of $\id j$ and $p(t)$ the Hilbert polynomial of $K[x_0,\mathbf x]/J$ with Gotzmann number $r$.
\begin{enumerate}[(i)]
\item \label{flat_i} $\underline{\mathrm{H}}_{\mathcal N(J)_r} = \underline{\mathrm{Mf}}_{\mathcal P(J_{\geq r})}\simeq \underline{\mathrm{Mf}}_{\mathcal P(\id j),r}$. 
\item \label{flat_ii} For every $m\geq 0$, $\underline{\mathrm{Mf}}_{\mathcal P(J_{\geq m})}$ and $\underline{\mathrm{Mf}}_{\mathcal P(\id j),m}$ are locally closed subfunctors of $\underline{\mathrm{Hilb}}_{p(t)}^n$.  If $m\geq \rho-1$, then  $\underline{\mathrm{Mf}}_{\mathcal P(J_{\geq m})}$ and $\underline{\mathrm{Mf}}_{\mathcal P(\id j),m}$ are open subfunctors of $\underline{\mathrm{Hilb}}_{p(t)}^n$.
\item \label{flat_iii} For every $m\geq 0$, the scheme  $ \mathrm{Mf}_{\mathcal P(\id j),m}=\Spec(K[C]/\mathfrak U)$ that represents the functor $\underline{\mathrm{Mf}}_{\mathcal P(\id j),m}$  can be canonically embedded as a locally closed subscheme of  the Hilbert scheme $\mathrm{Hilb}_n^{p(t)}$. 
\end{enumerate}
\end{Proposition}

\begin{proof}
Item \eqref{flat_i} is a direct consequence of Lemma \ref{lemma:immersione} and  Proposition \ref{prop:ponte}.

For what concerns item \eqref{flat_ii}, it is enough to apply item \eqref{flat_i} and Proposition \ref{prop:rho}, recalling that $\underline{\mathrm{H}}_{\mathcal N(J)_r}$ is an open subfunctor of the Hilbert scheme.

Item \eqref{flat_iii} is a direct reformulation of item \eqref{flat_ii}. The embedding that we consider is the one of Remark \ref{rem:ident}.
\end{proof}

Consider $\id G$ as in \eqref{eq:MSwithParam} and $\widetilde{\id G} \subset (K[C]/\id U)[\mathbf x]$, where $\widetilde{\id G}$ is the $[\id j,m]$-marked basis obtained from $\id G$ replacing every coefficient $C_{\alpha\eta}$ by its image in $K[C]/\mathfrak U$.
Observe that by  Remark \ref{rk:NotaMF}   $ \widetilde{\id G}$ is the universal family of the functor  $\underline{\mathrm{Mf}}_{\mathcal P(\id j),m}$.

\begin{theorem}\label{th:flatness}
Let $\id j$ be a quasi-stable ideal in $K[\mathbf x]$, $m$ a positive integer and $A$ a $K$-algebra.  
Every $K$-algebra morphism  $K[C]/\mathfrak U\rightarrow A$ defines the following flat family:
\[
\Proj\left( \left.\raisebox{.2em}{$(K[C]/\id U)[x_0,\mathbf x]$}\middle/\raisebox{-.2em}{$\left(\widetilde{\id G}\right)^h$}\right.\right)\times_{\Spec(K[C]/\id U)} \Spec(A)\rightarrow \Spec(A).
\]
\end{theorem}

\begin{proof} 
Let $\phi: K[C]/\id U\rightarrow A$ be a morphism of $K$-algebras. By composing the corresponding morphism of schemes with the embedding  $\Spec(K[C]/\id U)\hookrightarrow \mathrm{Hilb}_{p(t)}^n$ of Proposition \ref{prop:immersione} \eqref{flat_iii}, we get a morphism $\widehat\phi:\Spec(A)\rightarrow \mathrm{Hilb}_{p(t)}^n$. By definition of scheme representing a functor, $\mathrm{Hom}(\Spec(A),\mathrm{Hilb_{p(t)}^n})=\underline{\mathrm{Hilb}}_{p(t)}^n(A)$. Hence, $\hat\phi$ corresponds to a family $X_{\widehat \phi}\subseteq \mathbb P^n_A$ which is flat over $\Spec(A)$, by \eqref{eq:defH}.

More precisely, if we consider $A=K[C]/\id U$, $\phi=\mathrm{Id}$, the universal family $\widetilde {\id G}$ of the functor $\underline{\mathrm{Mf}}_{\mathcal P(\id j),m}$ and the projection morphism
\[
(K[C]/\id U)[x_0,\mathbf x]\rightarrow \left.\raisebox{.2em}{$(K[C]/\id U)[x_0,\mathbf x]$}\middle/\raisebox{-.2em}{$\left(\widetilde{\id G}\right)^h$}\right. , 
\]
we obtain the following flat family on $\Spec(K[C]/\id U)$
\begin{center}
% DA REINSERIRE
\begin{tikzcd} X_{\widehat{\mathrm{Id}}}=\Proj\left( \left.\raisebox{.2em}{$(K[C]/\id U)[x_0,\mathbf x]$}\middle/\raisebox{-.2em}{$\left(\widetilde{\id G}\right)^h$}\right.\right)\arrow[rd] \arrow[r,hook] & \PP^n_K\times\Spec(K[C]/\id U)\arrow[d]\\ & \Spec(K[C]/\id U) 
\end{tikzcd}
\end{center}
where $\times$ denotes the fibered product over $\Spec(K)$ and the vertical arrow is the projection on the second factor. 
For every  $K$-algebra morphism $\phi: K[C]/\id U\rightarrow A$, we obtain  $X_{\widehat{\phi}}=X_{\widehat{\mathrm{Id}}}\times_{\Spec(K[C]/\id U)} \Spec(A)$.
\end{proof}

The result of Theorem \ref{th:flatness} is not obvious. In the following example, we exhibit a family of ideals parameterized over an affine line, which is not flat although the ideals share the same affine Hilbert polynomial. 

\begin{Example}\label{ex:no flat}
Consider $K=\mathbb C$, $n=5$. Let $U_T$ be the set containing the following polynomials in $K[\mathbf x,T]=\mathbb C[x_1,\cdots,x_5,T]$:
\vskip 2mm
$\id f_1:=144\,x_{{1}}^{2}+284\,x_{{1}}x_{{2}}-317\,x_{{1}}x_{{3}}-212\,x_{{2}
}x_{{3}}+72\,x_{{3}}^{2}-43\,x_{{3}}x_{{4}},$

$\id f_2:= x_1x_4-x_3x_4-x_1x_2+x_2x_3,$

$\id f_3:=x_{{1}}x_{{2}}-x_{{1}}x_{{3}}-2\,x_{{2}}x_{{3}}+x_{{2}}x_{{4}}+x_{{3}}
x_{{4}},$

$\id f_4:=-4\,x_{{1}}x_{{2}}+x_{{1}}x_{{3}}+4\,x_{{1}}x_{{5}}+4\,x_{{2}}x_{{3}}-
5\,x_{{3}}x_{{4}}
,$

$\id f_5:=4\,x_{{1}}x_{{2}}-5\,x_{{1}}x_{{3}}-8\,x_{{2}}x_{{3}}+4\,x_{{2}}x_{{5}
}+5\,x_{{3}}x_{{4}},$

$\id f_6:=-4\,x_{{1}}x_{{2}}+3\,x_{{1}}x_{{3}}+4\,x_{{2}}x_{{3}}-7\,x_{{3}}x_{{4
}}+4\,x_{{3}}x_{{5}}
,$

$\id f_7:=-90\,x_{{1}}^{2}-56\,x_{{1}}x_{{2}}+71\,x_{{1}}x_{{3}}+18\,x_{{2}}
^{2}-34\,x_{{2}}x_{{3}}+109\,x_{{3}}x_{{4}}+18\,x_{{4}}^{2}
,$

$\id f_8:=4\,x_{{1}}x_{{2}}-5\,x_{{1}}x_{{3}}-4\,x_{{2}}x_{{3}}+x_{{3}}x_{{4}}+4
\,x_{{4}}x_{{5}}
,$

$\id f_9:=48\,x_{{1}}^{2}+68\,x_{{1}}x_{{2}}-83\,x_{{1}}x_{{3}}-44\,x_{{2}}x_{
{3}}-25\,x_{{3}}x_{{4}}+12\,x_{{5}}^{2}
,$

$\id f_{10}(T):=x_{{3}}x_{{4}}-x_{{2}}x_{{3}}+T \left( x_{{1}}^{3}+x_{{2}}^{3}+x_
{{2}}^{2} \right) .
$
\vskip 2mm

\noindent Let $\id i_T \subseteq K[\mathbf x,T]$ be the ideal generated by $U_T$ and, for every $\tau \in K$, let $\id i_\tau\subseteq \mathbb K[\mathbf x]$ be the ideal generated by the set $U_\tau$ obtained specializing $T$ to $\tau$. 

The fibers of the morphism  $Z=\Spec (K[\mathbf x,T]/\id i_T)  \rightarrow \mathbb A^1_K $ induced by  the embedding  $K[T] \hookrightarrow K[\mathbf x,T]/\id i_T$, form a family in the sense of \cite[Chapter II, Section 3, Definition at page 89]{H} and \cite[beginning of Chapter 6]{Eisenbud}. 
For every $\tau\in \mathbb A^1_K$, the fiber $Z_\tau$ over $\tau$ is the scheme $\mathrm{Spec}(K[\mathbf x]/\id i_\tau)$. 

By direct computations, we check that for all $\tau \in K$ the fibers $Z_\tau$  have Hilbert polynomial 12, hence correspond to points of the same Hilbert scheme  $\hilb^5_{12}$. Moreover, the ideals $\id i_\tau$ belong to $\Mf(\mathcal P(\id j),3)$, where $\id j$ is the strongly stable ideal generated by the following terms: $\quad
%\begin{equation*}
%\begin{split}
x_5^2,x_{{4}}x_{{5}},x_{{4}}^{2}, $ $x_{{3}}x_{{5}},x_{{3}}x_{{
4}},x_{{3}}^{2},x_{{2}}x_{{5}},x_{{2}}x_{{4}},x_{{1}}x_{{5}},x_{{1}}
x_{{4}},x_{{2}}^{2}x_{{3}},x_{{2}}^{3},x_{{1}}x_{{2}}x_{{3}},x_{{1}}x_{{
2}}^{2},x_{{1}}^{2}x_{{3}},x_{{1}}^{2}x_{{2}},$ $x_{{1}}^{4}.$
%\end{split}\end{equation*}

\noindent However, the family $Z \rightarrow \mathbb A^1_K $ is not flat. We can prove this fact  in several ways. 

For instance, we can apply the criterion of Artin described in \cite[Corollary to Proposition 3.1]{Artin} (see also \cite{BaMu}). In fact, if we take the restrictions of the syzygies $(h_1,\dots,h_{10})\in K[\mathbf x,T]^{10}$ of the polynomials $\id f_1,\dots,\id f_{10}(T)$ at  $T=0$, we see  that they do not generate the module of syzygies of the polynomials $\id f_1,\dots, \id f_{10}(0)$. For example,   $(0,0,0,0,x_4,-x_2+x_4,0,-x_2,0, 2x_1+2x_4-4x_5)$ is a  syzygy of the polynomials $\id f_1,\dots,\id f_{10}(0)$ that does not lift to a syzygy of   $\id f_1,\dots,\id f_{10}(T)$. 

Moreover, it is noteworthy that, for $\tau=0$, $\Proj(K[x_0,\mathbf x]/(\id i_0)^h)$ is a Gorenstein scheme in $\mathbb P^5_{K}$, while for other values of $\tau$ in $K$ the scheme $\Proj(K[x_0,\mathbf x]/(\id i_\tau)^h)$ is not Gorenstein. Since Gorenstein schemes constitute an open subset of $\hilb^5_{12}$ \cite[Theorem 3.31]{IE78}, this fact confirms that $Z  \rightarrow \mathbb A^1_K$ is not flat.

Finally, let $Y$ be the set of points in $\hilb^5_{12}$ that correspond to the $K$-fibers of the morphism $Z \rightarrow \mathbb A^1_K $, namely the schemes $\Proj(K[x_0,\mathbf x]/(\id i_\tau)^h)$ for $\tau \in K$. We see that $Y$ is embedded in $\hilb^5_{12}$ by a parameterization with parameter space $\{P\} \cup  (\mathbb A^1_{K} \setminus \{0\})$, where $P$ is an isolated point. Indeed, we compute the $[\mathcal P(\id j),3]$-marked bases of the ideals $\id i_\tau$ and find out that $Y$ is the set of $K$-points of a reducible scheme that is the union of two irreducible components. The first component is an isolated point $\{P\}$ and corresponds to the case $\tau=0$. The second component is isomorphic to $\mathbb A^1_{K} \setminus \{0\}$ and corresponds to the cases $\tau\in K\setminus\{0\}$. Then, we apply Theorem  \ref{th:flatness} and obtain a flat family whose $K$-fibers are the schemes defined by the ideals $\id i_\tau$, but the parameter space is $\{P\} \cup  (\mathbb A^1_{K} \setminus \{0\})$ and is not $\mathbb A^1_{K}$.
%
%$\{ \cap  \Mf(\mathcal P(\id j),3)\hookrightarrow \hilb^5_{12}$ cannot be a morphism of schemes. 
%{\color{blue}Da conservare commentato:} in realt\`{a}, calcoliamo esplicitamente la base di Groebner per $(U_T)$ con un ordinamento pesato e teniamo conto del fatto che una base di Groebner \`{e} anche una base marcata. L'ideale iniziale di $\id i_\tau$, $\tau \in \mathbb C^\ast$, \'e proprio $\id j$. Lo stesso ideale iniziale vale per $\tau=0$, ma calcolare la base di Groebner e specializzare il parametro $T$ a $0$ non commutano.
%\bcR 
%se considero nell'omogeneo l'ideale generato da $(U_T)^h$, non definisce una famiglia piatta, poich\'e per $T=0$, $(U_T)^h|_{T=0}$ definisce uno schema con polinomio di Hilbert 13. La famiglia non \'e piatta nell'omogeneo (Hartshorne) e quindi non \'e piatta sotto.
%\ecr
\end{Example}

%%%%%%%%%%%%%%%%
%% Algorithms %%
%%%%%%%%%%%%%%%%

\section{Algorithms}\label{algoritmi}

In this section, we describe an algorithm for computing the affine scheme representing the functor $\underline{\mathrm{Mf}}_{\mathcal P(\id j),m}$. 
Recall that if $\id j$ is a quasi-stable ideal, then its Pommaret basis can be explicitly computed, for instance, by \cite[Algorithm Involutive Completion]{GB}. Furthermore, the regularity $\reg(\id j)$ and the satiety $\sat(\id j)$ can be computed by Lemma \ref{lemma:importante} \eqref{importante_00} and \eqref{importante_0}.
Let us suppose that the following functions are available:

\noindent$\bullet$ \textsc{Coeff}$(g,x^\gamma)$. It returns the coefficient of the monomial $x^\gamma$ in the polynomial $g$. 

\noindent$\bullet$ $\textsc{LowerPart}(g,t)$. Given a polynomial $g$ and a non-negative integer $t$, it returns the pair of polynomials $(g_1,g_2)$ such that $g=g_1+g_2$, $\deg(g_1)\leq t$ and, for every $x^\gamma \in\supp(g_2)$, $\vert\gamma\vert\geq t+1$.

\noindent$\bullet$ $\textsc{Reduction}(g, \id G)$. Given a $[\id j,m]$-marked set $\id G$ and a polynomial $g$, it returns the polynomial $h$ such that  $g \SGred_\ast h$, $h\in \langle  \cN(\id j)\rangle$ (according to Definition \ref{relazione affine}).

\noindent$\bullet$ $\textsc{LowTerms}(\id j,m)$. It determines the set of terms in $\id j_{\leq m}$.

\begin{algorithm}[H]
\begin{algorithmic}[1]
\STATE $\textsc{Reduction1}(\id j, m, G)$
\REQUIRE $\id j \subset K[\bf x]$ quasi-stable ideal, $m$ a non-negative integer, $\id G$ a $[\mathcal P(\id j),m]$-marked set.
\ENSURE conditions to impose on the polynomials in $\id G$   to fulfill Theorem \ref{critsupermin} (\ref{critsupermin_i}).
\STATE $\textsf{Equations1} \leftarrow \emptyset$;
\STATE $B\leftarrow \textsc{LowTerms}(\id j,m)\setminus \mathcal P(\id j)$;
\FORALL{$x^\beta \in B$}
\STATE $h\leftarrow\textsc{Reduction}(x^\beta,\id G)$;
\STATE $(h_1,h_2)\leftarrow\textsc{LowerPart}(h,m)$;
\FORALL{$x^\gamma \in \supp(h_2)$}
\STATE $\textsf{Equations1}\leftarrow \textsf{Equations1}\cup \{\textsc{Coeff}(h_2,x^\gamma)\}$;
\ENDFOR
\ENDFOR
\RETURN $\textsf{Equations1}$;
\end{algorithmic}
\end{algorithm}

\begin{algorithm}[H]
\begin{algorithmic}[1]
\STATE $\textsc{Reduction2}(\id j, m, G)$
\REQUIRE $\id j \subset K[\bf x]$ quasi-stable ideal, $m$ a non-negative integer, $\id G$ a $[\mathcal P(\id j),m]$-marked set.
\ENSURE  conditions to impose on the polynomials in $\id G$   to fulfill Theorem \ref{critsupermin} (\ref{critsupermin_ii}).
\STATE $\textsf{Equations2} \leftarrow \emptyset$;
\FORALL{$f_\alpha \in \id G, x_i > \min(x^\alpha)$}
\STATE $h\leftarrow\textsc{Reduction}(x_if_\alpha,\id G)$;
\FORALL{$x^\gamma \in \supp(h)$}
\STATE $\textsf{Equations2}\leftarrow \textsf{Equations2}\cup \{\textsc{Coeff}(h,x^\gamma)\}$;
\ENDFOR
\ENDFOR
\RETURN $\textsf{Equations2}$;
\end{algorithmic}
\end{algorithm}

\begin{algorithm}[H]\label{algschemaaffine}
\begin{algorithmic}[1]
\STATE $\textsc{MarkedScheme}(\id j, m)$
\REQUIRE $\id j \subset K[x_1,\ldots,x_n]$ quasi-stable ideal, $m$ a non-negative integer.
\ENSURE an ideal defining the marked scheme $\Mf(\id j, m)$.
\STATE \label{line:minm}$m_0\leftarrow \min\{m,\sat(\id j)-1\}$;
\STATE $\id G \leftarrow \emptyset$;\label{line:startDefMSet}
\STATE $C\leftarrow \emptyset$;
\FORALL{$x^\alpha \in \mathcal P(\id j)$}
\STATE $f_{\alpha} \leftarrow x^\alpha$;
\FORALL{$x^\eta \in \mathcal N(\id j_{\leq \max\{m_0,\vert\alpha\vert\}})$}
\STATE $f_{\alpha} \leftarrow f_{\alpha} + {C}_{\alpha\eta} x^\eta$;
\STATE $C\leftarrow C\cup \{C_{\alpha\gamma}\}$;
\ENDFOR
\STATE ${\id G} \leftarrow \id G\cup \{f_{\alpha}\}$;
\ENDFOR\label{line:endDefMSet}
\STATE $\textsf{Equations} \leftarrow \emptyset$\label{line:defEq}
\STATE $\textsf{Equations}\leftarrow \textsc{Reduction2}(\id j,m_0,\id G)$\label{line:condSpoly}
\IF{$m_0{=} \reg(\id j)-1$ \AND $\mathcal N(\id j_{\leq \reg(\id j)-1})=\mathcal N(\id j_{\leq \reg(\id j)})$}\label{line:check0dim}
\RETURN $C,\textsf{Equations}$\label{line:firstReturn}
\ENDIF\label{line:endCheckdim0}
\STATE $\textsf{Equations}\leftarrow \textsf{Equations} \cup \textsc{Reduction1}(\id j,m_0,\id G)$\label{line:condjm}
\RETURN $C,\textsf{Equations}$\label{line:2return}
\end{algorithmic}
\end{algorithm}

\begin{theorem}\label{th:algoritmo}
Algorithm $\textsc{MarkedScheme}(\id j, m)$ returns the set $C$ of variables of a polynomial ring over $K$ and the generators of the ideal $\id U$ which defines the representing scheme $\Spec(K[C]/\id U)$ of the functor $\underline{\mathrm{Mf}}_{\mathcal P(\id j),m}$.
\end{theorem}

\begin{proof}
We analyze the command lines of Algorithm $\textsc{MarkedScheme}(\id j, m)$.

At line \ref{line:minm}, if $m\geq \rho-1$ then the algorithm computes the scheme representing $\underline{\mathrm{Mf}}_{\mathcal P(\id j),\rho-1}$ which is isomorphic to that representing $\underline{\mathrm{Mf}}_{\mathcal P(\id j),m}$ by Proposition \ref{prop:rho}.  
If $m<\rho-1$, then the algorithm computes exactly the scheme representing $\underline{\mathrm{Mf}}_{\mathcal P(\id j),m}$.

From line \ref{line:startDefMSet} to line \ref{line:endDefMSet}, the algorithm constructs the $[\id j,m_0]$-marked set $\id G$ as in \eqref{eq:MSwithParam} and the set of parameters as in \eqref{eq:parameters}.

At line \ref{line:condSpoly}, using Algorithm $\textsc{Reduction2}$, the algorithm computes and collects in the set \textsf{Equations}, which has been initialized at line \ref{line:defEq}, the coefficients belonging to $K[C]$ of the polynomials $h_{i\alpha}$, for every $x^\alpha \in \mathcal P(\id j)$ and every $x_i>\min(x^\alpha)$. 

From line \ref{line:check0dim} to line \ref{line:endCheckdim0}, the algorithm checks if $\mathcal N(\id j)$ is finite. If this is the case, the sets $C$ and \textsf{Equations} are returned (line \ref{line:firstReturn}).  Thanks to Proposition \ref{critpunti} and Theorem \ref{thm:schemamarcatoA}, the scheme $\Spec(K[C]/(\textsf{Equations}))$ represents $\underline{\mathrm{Mf}}_{\mathcal P(\id j),m_0}$.

Else, at line \ref{line:condjm} the algorithm also computes and adds to the set \textsf{Equations} the coefficients in $K[C]$ of the polynomials $h_{\beta}^{(>m)}$, for every $x^\beta \in \id j_{\leq m}\setminus \mathcal P(\id j)$, using Algorithm $\textsc{Reduction1}$.

Finally, at line \ref{line:2return} the algorithm returns the sets $C$ and \textsf{Equations}, which define the scheme $\Spec(K[C]/(\textsf{Equations}))$ representing $\underline{\mathrm{Mf}}_{\mathcal P(\id j),m_0}$ by Theorem \ref{thm:schemamarcatoA}.
\end{proof}

In next example we explicitly investigate some features of a Hilbert scheme of curves by the techniques developed in the previous sections, hence by applying our algorithm \textsc{MarkedScheme}. 

\begin{Example} \label{ex:tommasino}
Let us consider the Hilbert scheme $\HilbScheme{3t+2}{3}$ over a field of characteristic zero. There are $4$ saturated Borel-fixed ideals corresponding to points on this Hilbert scheme, which can be computed by the algorithm presented in \cite{CLMR}, improved in \cite{L} and generalized in \cite{B14}:
$$\mathfrak b_1:=(x_{{3}},x_{{2}}^{4},x_{{1}}^{2}  x_{{2}}^{3} ), \quad \mathfrak  b_2:=(x_{{3}}^{2},x_{{2}}x_{{3}},x_{{1}}x_{{3}},x_{{2}}^{4},x_{{1}} x_2^3),$$
$$\mathfrak  b_3:=( x_{{3}}^{2},x_{{2}}x_{{3}},x_{{2}}^{3},x_{{1}}^{2}x_{{3}}), \quad\mathfrak  b_4:=( x_{{3}}^{2},x_{{2}}x_{{3}},x_{{2}}^{3},x_{{1}}x_{{2}}^{2}).$$
The ideal $\mathfrak b_1$ corresponds to the lex-segment point of $\HilbScheme{3t+2}{3}$. It is well-known that such a point belongs to a unique component, which is rational. We denote this component by $Y_1$. By a direct computation, we find that  the dimension of $Y_1$ is $18$ and its general point corresponds to the union of a plane cubic curve and two isolated points. By \cite[Theorem 6]{RA}, also $\mathfrak b_2$ and $\mathfrak b_3$ define points of $Y_1$, while $\mathfrak b_4$ does not.

By our methods, we explicitly construct the open subset  $\Mf(\mathcal P(\mathfrak b_4),2)$ of $\HilbScheme{3t+2}{3}$ as an affine subscheme of $\mathbb A^{38}_K$, defined by $53$ equations which can be computed by the Algorithm \textsc{MarkedScheme}. By a direct computation on the ideal generated by these $53$ equations, we find that $\mathfrak b_4$ is contained in two irreducible components $Y_2$ and $Y_3$. We find that the dimension of $Y_2$ is $12$ and its general point corresponds to the disjoint union of a conic and a line, while the dimension of $Y_3$ is $15$ and its general point corresponds to the union of a twisted cubic curve and a point.
For more details about $Y_2$ and $Y_3$ we refer to \cite[Example 6.8]{BCR}.
\end{Example}

\subsection{Comparison with similar methods}
\label{subsec:comparison}

In summary, our algorithm \textsc{MarkedScheme} essentially consists in a rewriting procedure particularly suitable for the computation of the scheme parameterizing the family of marked bases over a given 
quasi-stable ideal $\mathfrak j\subset \Ax$. The computation depends on the Pommaret basis of the ideal $\mathfrak j$ and on a given integer $m$ that guarantees that this marked family is parameterized by a finite number of parameters. Recall that this family is embedded in a Hilbert scheme (see Proposition \ref{prop:immersione}). 

There are essentially two other rewriting procedures already used for the same purpose, which are described respectively in \cite{BCLR} and in \cite{CMR13}. More precisely, the method of \cite{BCLR}, called superminimal reduction, computes the same scheme as our algorithm \textsc{MarkedScheme} but only provided that $\mathfrak j$ is a strongly stable ideal (see Remark \ref{rem:funt}), constructing marked bases over $J_{\geq m}=(\mathfrak j)^h_{\geq m}\subset \Aox$. The method of \cite{CMR13} computes 
an affine scheme that parameterizes a $\mathcal P(J)$-marked family for every quasi-stable ideal $J\subset \Aox$. In particular, it is able to compute a scheme that is isomorphic to the scheme which is computed by \textsc{MarkedScheme}, because it constructs $\mathcal P(J_{\geq m})$-marked bases for every quasi-stable ideal $\mathfrak j$.

We now highlight the main differences between \textsc{MarkedScheme} and the method of \cite{CMR13}. First, the number of parameters which are involved in the computation of  the ideal defining the $\mathcal P(J_{\geq m})$-marked scheme following the lines of \cite{CMR13} is higher than the one in Algorithm \textsc{MarkedScheme}. Indeed, in order to compute the ideal defining the $\mathcal P(J_{\geq m}$)-marked scheme as in \cite{CMR13}, we need to use the set of parameters 
$$C'=\{C'_{\alpha\eta}\}_{x^\alpha\in \mathcal P(J_{\geq m}),x^\eta\in \mathcal N(J_{\geq m})_{\vert\alpha\vert}},$$ 
which in general strictly contains the set $C$ as in \eqref{eq:parameters}. 
Thus, the computation of \cite{CMR13} involves more variables, and the result coincides with that of \textsc{MarkedScheme} only after the elimination of the exceeding variables.
Second, the number of polynomials that have to be reduced by $\xrightarrow{ \ G^{(\ell)}\ }_\ast$ in order to define the scheme according to the technique of \cite{CMR13} is very high. Indeed, for every $x^\beta\in \id j_{\leq m}\setminus \mathcal P(\id j)$, such that $\vert\beta\vert<m$, there is $x^{\beta'}=x^\beta x_0^t\in \mathcal P(J_{\geq m})$, with $t=m-\vert\beta\vert\geq 1$, by Proposition \ref{prop:PBtronc}\eqref{prop:PBtronc_i}. Hence, the algorithm suggested by \cite[Definition 6.1]{CMR13}, according to Theorem \ref{th:varie CMR} \eqref{varie CMR_iv}, computes the polynomial reduction of $n$ polynomials $x_iF_{\beta'}$, $i\in \{1,\dots,n\}$, while in Algorithm \textsc{MarkedScheme}, only the reduction by $\xrightarrow{\ \id G_{\mathcal P} }_\ast$  of the term $x^\beta$ is computed. 

On the other hand, it may happen that the algorithm of \cite{CMR13} is able to conclude its computation  in a very larger affine space, while in the same case the algorithm \textsc{MarkedScheme} does not return the output in a reasonable time (see Example \ref{ex:confronti}). However, \textsc{MarkedScheme} intrinsically involves also the elimination procedure of the unnecessary parameters, hence a direct comparison of the execution timings with the algorithm in \cite{CMR13} should include the elimination of variables from the output of \cite{CMR13}. 

In conclusion, either one of the two algorithms can be more advantageous, depending on the purpose, as we observe in the following example in which we consider again the study described in Example \ref{ex:tommasino}, but from different points of view. 

\begin{Example}\label{ex:confronti}
For the Hilbert scheme $\hilb_{3t+2}$ considered in Example \ref{ex:tommasino}, we succeeded in completing the computation of the marked family over the ideal $\mathfrak b_1$ corresponding to the lex-segment point by using an implementation of the method described in \cite{CMR13} for the strongly stable case. But, we succeeded in the detection of the two irreducible components $Y_2$ and $Y_3$ of the marked family over the ideal $\mathfrak b_4$ thanks to the more {\em concise} representation of the marked scheme obtained by the procedure \textsc{MarkedScheme}. Moreover,
for the ideals $\mathfrak b_1$ and $\mathfrak b_2$ we obtain the following performance, where:
\begin{itemize}
\item the execution timings are in seconds;
%\item $T$ is the tangent space of the marked scheme at the origin;
\item NF means \lq\lq after more than $90$ hours we stopped the execution";
\item \lq\lq construction" is the execution of the algorithm;
\item \lq\lq elimination" means the elimination of the exceeding variables and/or of possible other eliminable variables after the construction.
\end{itemize} 
We used Maple 18 under Windows 7 professional (64 bit), on a computer with a 3.50 GHz Intel Xeon processor and 8 Gb RAM:
$$
\begin{tabular}{l}
\begin{tabular}{l|r|r|r|r|r}
\multicolumn{6}{c}{MS=\textsc{MarkedScheme}, CMR= method of \cite{CMR13}}
\vspace{2mm}
\\
\hline
& $m$ & $\#$ param & $\#$ eq &construction & elimination \\
\hline
$\mathfrak b_1$ MS & 4&  45&  &  NF &  \\
$\mathfrak b_1$ CMR& 4&  311& 822&  0.9 & NF \\
\hline
$\mathfrak b_2$ MS & 3&  61& 126&  0.3 & 19.5 \\
$\mathfrak b_2$ CMR& 3& 127& 289&  0.3 &12767.7\\
%\hline
%$\mathfrak b_3$ MS & &  & &  &  &  \\
%$\mathfrak b_3$ CMR& &  & &  &  &  \\
%\hline
%$\mathfrak b_4$ MS & &  & &  &  &  \\
%$\mathfrak b_4$ CMR& &  & &  &  &  \\
\hline
\end{tabular}
\end{tabular}
$$
\end{Example}

\medskip
Concerning the method of \cite{BCLR}, first we observe that we work under the more general hypothesis that $\id j$ is a quasi-stable ideal, while the results of \cite{BCLR} are proved under the hypothesis that the considered monomial ideals are strongly stable. Anyway, the difference between \textsc{MarkedScheme} and the method of \cite{BCLR} essentially consists in the fact that the reduction process $\xrightarrow{\ \id G_{\mathcal P}\ }$ that we use to construct the scheme is much simpler to implement and execute than the superminimal reduction for homogeneous polynomials. Indeed, each step of superminimal reduction needs a multiplication by a power of the smallest variable $x_0$ (see \cite[Definition 3.2 and Theorem 3.5]{BCLR} for details). 
The exponent of this power can be very large and the subsequent multiplications might become a bottleneck for the time of execution of algorithms using this polynomial reduction. In the affine setting this problem corresponds to the fact that, in the course of the reduction, polynomials with very high degree  can be involved. This problem is strictly related to an evaluation of the computational cost that we address in the next subsection. 

\subsection{Some complexity results}

Like the methods of \cite{BCLR,CMR13}, our algorithm is based on a rewriting procedure that for every polynomial $f$ recognizes the unique decomposition $f=\sum p_i g_i+\Nf(f)$ by using a marked basis $\{g_1,\dots,g_t\}$. Hence, as for the computation of Gr\"obner bases, the computational cost of these methods could be described in terms of the computational cost of the polynomial ideal membership problem (see \cite{BSsizigie,Mayr} and the references therein). In this context, the key point is in fact the knowledge of bounds on the degrees of the polynomials $p_i$ involved in the rewriting procedure (see \cite{Hermann,Lazard} and the references therein). Generally, these bounds depend on the maximal degree of the polynomials $g_i$ and $\sum p_i g_i$ and on the number $n$ of variables of the polynomial ring over a field. 

As it happens in our computation of marked schemes, we have to consider also the case in which the polynomials $g_1,\dots,g_t$ belong to $K[C][\mathbf x]$ and do not form a marked basis yet.
Anyway, even when $g_1,\dots,g_t$ form a marked basis, we could have some problems. We know that the degree of $\Nf(f)$ is $\leq\max\{\deg(f),m\}$ and so also the degree of $\sum p_i g_i$ is $\leq\max\{\deg(f),m\}$. However, as we have already outlined, during the reduction $\xrightarrow{\id G_{\mathcal P}}$, polynomials $p_ig_i$ of degree higher than $\max\{\deg(f),m\}$ can be involved, but this might be not apparent when looking at $\sum p_ig_i$, due to some term cancellations, as the following example shows.

\begin{Example}\label{nonomognonbasta}
Consider $\id j=(x_2,x_1^t)\subseteq K[x_1,x_2]$ and, for every integer $t\geq 2$, the polynomials $f_1:=x_2-cx_1^{t-1}$, with $c\in K^\ast$, and $f_2:=x_1^t$. The set $\mathfrak G:=\{f_1,f_2\}$ is a $[\id j,t]$-marked basis.
%l'ideale è "artiniano", quindi basta verificare che l'S-polinomio x_2*f2-x_1^t*f1 si riduca a 0 (e questo avviene sempre, per ogni t)
The ring $K[x_1,x_2]/\mathfrak j$ is Artinian, hence the normal form of every term in $\mathfrak j_{\leq t}$ is always null by Theorem \ref{critsupermin} and Proposition \ref{critpunti}. In particular, we have $\Nf(x_2^t)=0$ because  
$x_2^t=(\sum_{i=0}^{t-1} c^i x_2^{t-1-i} x_1^{i(t-1)}) f_1+c^t f_2^{t-1}$, where the last addend is $c^t x_1^{t(t-1)}$. Hence, the maximal degree in the variables $x_1,x_2$ is $t(t-1)$.
\end{Example}

Eventually, we obtain the following result that gives a bound for the maximal degree of the polynomials occuring in the reduction of a term $x^\beta$ by our method, more precisely for the maximal degree of the polynomials $p_ig_i$ when $\mathfrak G=\{g_1,\dots,g_t\}=\{f_{\alpha_i}\}_{x^{\alpha_i}\in \mathcal P(\mathfrak j)}$ is a $[\mathcal P(\mathfrak j),m]$-marked set. To reach this aim we will consider sequences of terms
%\begin{equation}\label{eq:sequence}
$\{x^{\eta_0},x^{\eta_1,}\dots,x^{\eta_s}\}$
%\end{equation} 
such  that for every $0 \leq i <s$
\begin{equation}\label{eq:condition}
x^{\eta_i}=x^{\delta_i} x^{\alpha_i} \in \mathcal C_{\mathcal P}(x^{\alpha_i}), \quad x^{\eta_{i+1}} \in \mathrm{supp}(x^{\eta_i}-x^{\delta_i}f_{\alpha_i}).
\end{equation}  

\begin{theorem}\label{th:bound}
Let $\mathfrak G\subset \Ax$ be a $[\mathcal P(\mathfrak j),m]$-marked set and $d$ the maximal degree of the polynomials in $\mathfrak G$. Then, the maximal degree of a polynomial occurring in the reduction of a term $x^\beta $ of degree $h$ by $\xrightarrow{\  \id G_{\mathcal P}\ }$ is lower than or equal to $(h-1)d^{n-1}+d$. 

Moreover, for every sequence $\{x^{\eta_i}\}_{i=0,\dots,s}$ satisfying condition \eqref{eq:condition} we obtain 
$$s \leq \left\{\begin{array}{ll} h &\text{ if } d=1 \\ \frac{(h-1)(d^n-1)}{d-1} &\text{ if } d>1. \end{array}\right.$$
\end{theorem}

\begin{proof}
By the hypothesis we have $d \leq \max\{m, \vert x^\alpha \vert : \ x^\alpha \in \mathcal P(\mathfrak j)\}$.
By Proposition \ref{th:reductionAffineSet}, we can take a sequence $\{x^{\eta_i}\}_{i=0,\dots,s}$ of terms satisfying condition \eqref{eq:condition}, i.e.~for every $0\leq i \leq s-1$, $x^{\eta_{i+1}}$ appears in the first step of reduction of $x^{\eta_i}$ by $\xrightarrow{\  \id G_{\mathcal P}\ }$ and belongs to $\mathfrak j$. Thus, for every $0\leq i \leq s-1$, there is a term $x^{\alpha_i}$ of the Pommaret basis $\mathcal P(\mathfrak j)$ such that $x^{\eta_i}=x^{\alpha_i} x^{\delta_i}\in \mathcal C_{\mathcal P}(x^{\alpha_i})$. We assume $x^{\eta_0}=x^\beta$.

By Lemma \ref{lemma:importante}\eqref{importante_iii}, we have $\max(x^{\delta_i})\geq \max(x^{\delta_{i+1}})$ and if the equality holds then the exponent of $\max(x^{\delta_i})$ in $x^{\delta_i}$ is higher than its exponent in $x^{\delta_{i+1}}$. 

Observe that if $d=1$ then all polynomials of $\mathfrak G$ have degree $1$ and $h \geq \vert x^{\eta_i}\vert$ for every $i\in \{0,\dots,s\}$ by construction. Moreover, in this case we obtain also $s\leq h$ by construction. Hence, let $d$ be $\geq 2$.

We first consider the special case $\max(x^{\delta_i})=x_v$ for every $i\in \{0, \dots, s-1\}$ and either $\max(x^{\delta_s})<x_v$ or $x^{\eta_s}\notin \id j$. Note that $s\leq \vert x^{\eta_0}\vert- \vert x^{\alpha_0}\vert$ because the exponent of $x_v$ in $x^{\delta_i}$ strictly decreases at every step. By construction we have $\vert x^{\eta_{i+1}}\vert \leq \vert x^{\eta_{i}}\vert +d-1$, for every $0 \leq i < s$. Thus, we obtain 
\begin{equation}\label{formulagrado}
\vert x^{\eta_s}\vert \leq \vert x^{\eta_0}\vert + (d-1 )( \vert x^{\eta_0}\vert- \vert x^{\alpha_0 }\vert)=  d\vert x^{\eta_0}\vert-(d-1)  \vert x^{\alpha_0 }\vert.
\end{equation}
We can summarize  the bounds for a sequence of the above special type as follows: 
\begin{enumerate}
\item \label{fatto1} the maximal degree of the terms  is lower than or equal to $d$ times the degree of the first term minus $d-1$; 
\item  \label{fatto2} $s \leq$ degree of the first term minus $1$.
\end{enumerate}
Let us now consider the general case. Observe that the sequence $\{x^{\eta_i}\}_{i=0,\dots,s}$ is the union of at most $n$ consecutive sub-sequences of the above special type, because we have $n$ variables and $\max(x^{\delta_i})\geq \max(x^{\delta_{i+1}})$, where if the equality holds then the exponent of $\max(x^{\delta_i})$ in $x^{\delta_i}$ is higher than its exponent in $x^{\delta_{i+1}}$. We assume $x^{\eta_s}\notin \mathfrak j$ because every other sequence can be extended to a sequence with this property.

For convenience, let $\ell_{n+1}:=0$ and $h_{n+1}=h$. Moreover, for every variable $x_v\geq x_1$, let ${\ell_v}:=\min\{i\in\{0,\dots,s\} \ : \ \max(x^{\delta_{i}})<x_{v} \text{ or } x^{\eta_i}\notin \mathfrak j\}$ and let us denote by $h_v$ the maximal degree of the terms $x^{\eta_i }$ with $i\leq \ell_v$. 
Note that $h=h_{n+1}\leq h_n \leq \dots \leq h_2\leq h_1$ and $\ell_1=s$. In particular, if $\max(x^\beta)=x_j$ then we have $h=h_{n+1}= \dots =h_{j+1}$. We first consider the sub-sequences of the above special type $\{x^{\eta_i}\}_{i=\ell_{v+1},\dots,\ell_v}$ for every $v \in \{2,\dots,n\}$.
By \eqref{fatto1} for every $v\in\{2,\dots,n\}$ we have $h_v \leq (h_{v+1}-1)d+1$ and by recursion
$$h_v \leq (h-1)d^{n-v+1} +1.$$
In particular, we obtain $h_2\leq (h-1)d^{n-1}+1$.

If $x^{\eta_{\ell_2}} \notin \id j$ then we can conclude.  If $ x^{\eta_{\ell_2}} \in \id j$ then we consider the last sub-sequence $\{x^{\eta_i}\}_{i=\ell_2,\dots,s}$. In this case, for every $\ell_2 \leq i\leq s-1$ we have $ x^{\delta_{i}}= x_{1}^{c_i}$ for some  non-negative integers $c_i $ with $c_{\ell_2} =\vert x^{\eta_{\ell_2}} \vert - \vert x^{\alpha_{\ell_2}} \vert\leq   h_2-1  $ and $c_{\ell_2}>c_{\ell_2+1} \dots >c_{s-1}$ by Proposition \ref{th:reductionAffineSet}. Thus for every $\ell_2 < i \leq s$ we have $\vert x^{\eta_{i}}\vert\leq d+c_{\ell_2}\leq d+ h_2-1$ by construction.

Therefore, the maximal degree of the terms in the whole sequence is lower than or equal to $(h-1)d^{n-1}+d$.

Let us now consider the number $s$ of steps of the reduction $\xrightarrow{\  \id G_{\mathcal P}\ }$ we need to obtain the term $x^{\eta_s}$ from the term $x^{\eta_0}$ of the sequence $\{x^{\eta_i}\}_{i=0\dots,s}$. By \eqref{fatto2}, for every $v\in \{2,\dots,n\}$ we consider the sub-sequence $\{x^{\eta_i}\}_{i=\ell_{v+1},\dots,\ell_v}$ and have $\ell_v- \ell_{v+1} \leq \vert x^{\eta_{\ell_{v+1}}} \vert -1 \leq h_{v+1}-1$. Moreover, we have already proved that $\ell_1-\ell_2=s-\ell_2 \leq h_2-1$. Then, $s\leq \sum_{v=1}^{n} (h_{v+1}-1)$ and we obtain the claimed bound.
\end{proof}

\begin{Remark}\label{rm:sharp}
We observe that the first bound of Theorem \ref{th:bound} is sharp by exhibiting the unique case in which the inequality is an equality, beyond the banal case $x^\beta=1 \in\mathfrak j$. 
We can have an equality only if:
\begin{itemize}
\item there are exactly $n$ sub-sequences of special type, namely $\ell_{v+1}\neq \ell_v$ for every $v > 1$;
\item we have an equality in \eqref{formulagrado} for every $v$, i.e.~$\vert x^{\alpha_i}\vert =1$ for every $i$.
\end{itemize}
Hence, $x_n, \dots, x_2$ must belong to $\id j$ and $\mathcal N(\id j)_{\leq d}\subseteq \{1, x_1, \dots, x_1^d\}$. Moreover,
\begin{itemize}
\item at every step, $x^{\eta_i}$ is replaced by a term in $\mathcal N(\id j)_{ d}$, i.e. by $x_1^{d}$.  
\end{itemize}
With these constraints, every variable in $x^{\eta_0}$ is replaced by $x_1^d$ and we obtain $hd$ as maximal degree. Therefore, the only cases in which we have an equality are those presented in next Example \ref{nonomognonbasta2} because the two integers $(h-1)d^{n-1}+d$ and $hd$ are equal only if $n=2$. 
\end{Remark}

\begin{Example}\label{nonomognonbasta2}
Let $h,d$ be any two positive integers. Consider $\id j=(x_2)\subseteq K[x_1,x_2]$ and the polynomial $f_1:=x_2-\sum_{i=0}^d  a_i x_1^{i}$ for some coefficient $a_i \in K$, $a_d\neq 0$. The set $G:=\{f_1\}$ is a $[\mathcal P(\id j),d]$-marked basis.
The normal form of the term $x_2^h$ is the polynomial  $(\sum_{i=0}^d  a_i x_1^{d_i})^h$ whose maximal degree is $hd=(h-1)d^{n-1}+d$, where $n=2$ is the number of variables. 
\end{Example}

Let $\mathfrak G\subset K[C][\mathbf x]$ be the $[\mathcal P(\mathfrak j),m]$-marked set which we defined in \eqref{eq:MSwithParam} and let $d$ be the maximal degree of the polynomials of $\mathfrak G$ in the variables $\mathbf x$.

\begin{Corollary}\label{cor:grado equazioni}
The maximal degree of the generators of the ideal $\mathfrak U$ that are computed by the algorithm \textsc{MarkedScheme} is $\leq \frac{d(d^n-1)}{d-1}$.
\end{Corollary}

\begin{proof}
At every step of the rewriting procedure $\xrightarrow{\  \id G_{\mathcal P}\ }$ that the algorithm \textsc{MarkedScheme} performs starting from the polynomials of $\mathfrak G$, the degree of the coefficients in the variables $C$ increases at most by $1$. Then, we can conclude because by Theorem \ref{th:bound} we have at most $\frac{(h-1)(d^n-1)}{d-1}$ steps, where $h$ is the maximal degree of the terms to be reduced, that in this case is $\leq d+1$ by Theorem \ref{critsupermin}. 
\end{proof}

\begin{Example}\label{nonomognonbasta3}
Let $h,d$ be any two positive integers.  Consider $\id j=(x_2,x_1^{d})\subseteq K[x_1,x_2]$ and  the polynomials $f_1:=x_2-x_1^{d-1}, f_2:=x_1^d-x_1^{d-1}$. The set $G:=\{f_1,f_2\}$ is a $[\mathcal P(\id j),d-1]$-marked basis with polynomials of maximal degree $d$.
In order to obtain the normal form of the term  $x_2^h$ we perform $h$ steps of reduction by $f_1$, getting $x_1^{h(d-1)}$, and then $h(d-1)-(d-1)$ steps of reduction by $f_2$ getting the normal form $x_1^{d-1}$. Therefore the total number of steps is $(h-1)d+1$. 
\end{Example}

\section*{Acknowledgements}
%The authors were partially supported by GNSAGA (INdAM, Italy). The second and third authors were partially supported by the framework of PRIN 2010-11 \lq\lq Geometria delle variet\`a\ algebriche\rq\rq, cofinanced by MIUR.

The authors thank the anonymous referees for their constructive criticism.
 
\bibliographystyle{abbrv}

\end{document}